\documentclass[12pt]{amsart}
\usepackage{amsmath,amsthm,amssymb}
\newtheorem{theorem}{Theorem}[section]
\newtheorem{corollary}[theorem]{Corollary}
\newtheorem{proposition}[theorem]{Proposition}
\newtheorem{lemma}[theorem]{Lemma}
\theoremstyle{definition}

\theoremstyle{remark}

\numberwithin{equation}{section}
\DeclareMathOperator{\trans}{\mathstrut^{\mathit{t}}\!}
\DeclareMathOperator{\Prym}{Prym}
\DeclareMathOperator{\Sym}{Sym}
\DeclareMathOperator{\Mat}{Mat}
\DeclareMathOperator{\Pic}{Pic}
\DeclareMathOperator{\NS}{NS}

\DeclareMathOperator{\Branch}{Branch}
\DeclareMathOperator{\Spec}{Spec}
\DeclareMathOperator{\Proj}{Proj}
\DeclareMathOperator{\Sp}{Sp}
\DeclareMathOperator{\Bs}{Bs}
\DeclareMathOperator{\supp}{supp}
\DeclareMathOperator{\Div}{Div}

\DeclareMathOperator{\Ker}{Ker}
\DeclareMathOperator{\Image}{Image}
\DeclareMathOperator{\im}{Im}

\DeclareMathOperator{\diag}{diag}
\DeclareMathOperator{\Nm}{Nm}
\DeclareMathOperator{\mult}{mult}

\title[Prym loci of double coverings and Andreotti-Mayer loci]
{Prym loci of branched double coverings and generalized Andreotti-Mayer loci}
\author{Atsushi Ikeda}
\address{Department of Mathematics and Data Science, School of Engineering,
Tokyo Denki University, Adachi-ku,
Tokyo 120-8551, Japan}
\email{atsushi@mail.dendai.ac.jp}
\thanks{The author is partially supported by JSPS KAKENHI Grant Number 20K03543.}
\begin{document}
\begin{abstract}
 The Andreotti-Mayer locus is a subset of the moduli space of principally polarized abelian varieties, defined by a condition on the dimension of the singular locus of the theta divisor.
 It is known that the Jacobian locus in the moduli space is an irreducible component of the Andreotti-Mayer locus.
 In this paper, we generalize the Andreotti-Mayer locus to the case of the moduli space of abelian varieties with non-principal polarization and prove that the Prym locus of branched double coverings is an irreducible component of the generalized Andreotti-Mayer locus.
\end{abstract}
\maketitle
\section{Introduction}
Let $C$ be a projective smooth curve of genus $g$ over the complex numbers $\mathbb{C}$, and let $\phi:\tilde{C}\rightarrow{C}$ be a double covering of $C$ branched at $2n$ points.
When $n\geq1$, the norm map $\Nm:\Pic{(\tilde{C})}\rightarrow\Pic{(C)}$ associated with $\phi$ has the connected kernel $P=\Ker{(\Nm)}$, which is an abelian subvariety of dimension $d=g+n-1$ in the Jacobian variety $J_{\tilde{C}}=\Pic^{0}{(\tilde{C})}$.
Let $\mathcal{L}\in\Pic{(P)}$ be the restriction of the invertible sheaf $\mathcal{O}_{J_{\tilde{C}}}(\Theta_{\tilde{C}})$ associated with the theta divisor $\Theta_{\tilde{C}}$ on $J_{\tilde{C}}$.
Then the class $[\mathcal{L}]\in{\NS{(P)}}$ determines the polarization on $P$ of type
$\Delta=(\underbrace{1,\dots,1}_{n-1},\underbrace{2,\dots,2}_{g})$,
and the polarized abelian variety $(P,[\mathcal{L}])$ is called the Prym variety (\cite{M2}) of the branched covering $\phi$.
From this construction, we obtain the Prym map
$$
\Prym_{g,2n}:\,\mathcal{R}_{g,2n}\longrightarrow\mathcal{A}_{d}^{\Delta};\,
[\tilde{C}\overset{\phi}{\rightarrow}{C}]\longmapsto
(P,[\mathcal{L}]),
$$
where $\mathcal{R}_{g,2n}$ denotes the moduli space of double coverings of curves of genus $g$ branched at $2n$ points, and $\mathcal{A}_{d}^{\Delta}$ denotes the moduli space of $d$-dimensional polarized abelian varieties of type $\Delta$.
When $g=0$, the Prym variety $\Prym_{0,2n}(\phi)=(J_{\tilde{C}},[\mathcal{O}_{J_{\tilde{C}}}(\Theta_{\tilde{C}})])$ is the hyperelliptic Jacobian, and the Prym map $\Prym_{0,2n}$ is injective by the Torelli theorem.
When $g=1$, it is proven in \cite{I} that the Prym map $\Prym_{1,2n}$ is injective for $n\geq3$.
More generally, the following result holds:
\begin{theorem}[Naranjo-Ortega \cite{NO2}]\label{185532_14Mar25}
 If $g\geq1$ and $n\geq3$, then the Prym map $\Prym_{g,2n}$ is injective.
\end{theorem}
Let $\mathcal{P}_{g,2n}\subset\mathcal{A}_{d}^{\Delta}$ be the closure of the image of the Prym map $\Prym_{g,2n}$.
We aim to understand which kinds of polarized abelian varieties are contained in the Prym locus $\mathcal{P}_{g,2n}$.
In the case where the polarization is principal, the characterization of the Jacobian locus has been classically studied as the Schottky problem.
Let $\mathcal{J}_{d}\subset\mathcal{A}_{d}^{(1,\dots,1)}$ (resp.\ $\mathcal{H}_{d}\subset\mathcal{A}_{d}^{(1,\dots,1)}$) denote the closure of the set of points representing the Jacobian varieties of projective smooth curves (resp.\ hyperelliptic curves) of genus $d$.
In this paper, we characterize $\mathcal{P}_{g,2n}$ in $\mathcal{A}_{d}^{\Delta}$, following the approach of Andreotti and Mayer \cite{AM} to characterize $\mathcal{J}_{d}$ and $\mathcal{H}_{d}$ in $\mathcal{A}_{d}^{(1,\dots,1)}$.\par
For an ample invertible sheaf $\mathcal{L}$ on an abelian variety $A$, we define the higher base locus of $\mathcal{L}$ by
$$
S(A,\mathcal{L})=\bigcap_{\Theta\in|\mathcal{L}|}\Theta_{\mathrm{sing}}
\subset{A},
$$
where $\Theta_{\mathrm{sing}}$ denotes the singular locus of the hypersurface $\Theta\subset{A}$.
If ample invertible sheaves $\mathcal{L}_{1}$ and $\mathcal{L}_{2}$ are algebraically equivalent,
then $S(A,\mathcal{L}_{1})\simeq{S(A,\mathcal{L}_{2})}$,
because $t_{x_{0}}^{*}\mathcal{L}_{1}\simeq\mathcal{L}_{2}$ by a translation
$$
t_{x_{0}}:A\longrightarrow{A};\,x\longmapsto{x+x_{0}}.
$$
For an integer $m\geq0$, a generalization of the Andreotti-Mayer locus is defined by
$$
\mathcal{N}_{d,m}^{\Delta}
=\{(A,[\mathcal{L}])\in\mathcal{A}_{d}^{\Delta}\mid
\dim{S(A,\mathcal{L})}\geq{m}\}.
$$
When $\Delta=(1,\dots,1)$, the subset
$\mathcal{N}_{d,m}^{(1,\dots,1)}\subset\mathcal{A}_{d}^{(1,\dots,1)}$
is the original Andreotti-Mayer locus, and the following results are known:
\begin{theorem}[Andreotti-Mayer \cite{AM}]\label{164133_28Jan25}
 \begin{enumerate}
  \item If $d\geq4$, then the Jacobian locus $\mathcal{J}_{d}$ is an irreducible component of $\mathcal{N}_{d,d-4}^{(1,\dots,1)}$.
  \item\label{164150_28Jan25}
       If $d\geq3$, then the hyperelliptic Jacobian locus $\mathcal{H}_{d}$ is an irreducible component of $\mathcal{N}_{d,d-3}^{(1,\dots,1)}$.
 \end{enumerate}
\end{theorem}
\begin{theorem}[Debarre \cite{D}]
 If $d\geq7$, then $\mathcal{P}_{d+1,0}$ is an irreducible component of $\mathcal{N}_{d,d-6}^{(1,\dots,1)}$, where $\mathcal{P}_{d+1,0}\subset\mathcal{A}_{d}^{(1,\dots,1)}$ denotes the Prym locus for unramified double coverings of curves of genus $d+1$.
\end{theorem}
In this paper, we prove the following:
\begin{theorem}\label{165539_28Jan25}
 If $n\geq4$ and $d=g+n-1$, then the Prym locus $\mathcal{P}_{g,2n}\subset\mathcal{A}_{d}^{\Delta}$ is an irreducible component of $\mathcal{N}_{d,n-4}^{\Delta}$.
\end{theorem}
This is a generalization of Theorem~\ref{164133_28Jan25}~(\ref{164150_28Jan25}),
because $\mathcal{P}_{0,2n}=\mathcal{H}_{d}\subset\mathcal{A}_{d}^{(1,\dots,1)}$ for $d=n-1$.
According to the folk conjecture \cite[Conjecture~3.15~(2)]{GH},
it is expected that
$\overline{\mathcal{N}_{d,d-3}^{(1,\dots,1)}
\setminus\mathcal{N}_{d,d-2}^{(1,\dots,1)}}
=\mathcal{H}_{d}$.
More generally we expect that
$\overline{\mathcal{N}_{d,n-4}^{\Delta}\setminus\mathcal{N}_{d,n-3}^{\Delta}}=\mathcal{P}_{g,2n}$ for $d=g+n-1$,
but this remains an open problem.\par
This paper proceeds as follows:
In Section~\ref{094948_19Mar25}, for the Prym variety $(P,[\mathcal{L}])$ of a branched double covering, we explicitly describe the higher base locus $S(P,\mathcal{L})\subset{P}$, and prove that
$\mathcal{P}_{g,2n}\subset\mathcal{N}_{d,n-4}^{\Delta}$.
In the proof, we provide the equation of the tangent cone
$\mathcal{T}_{\Theta,x}\subset{T_{P,x}=T_{P,0}}$ of $\Theta\in|\mathcal{L}|$
at $x\in{S(P,\mathcal{L})}$.
In Section~\ref{165857_20Mar25}, we prove that
the equations of the tangent cones $\mathcal{T}_{\Theta,x}$
generate a certain subspace in $\Sym^{2}{T_{P,0}}^{\vee}$.
In Section~\ref{165927_20Mar25}, following the approach in \cite{AM}, we introduce a method to bound the dimension of an irreducible component $\mathcal{P}$ of $\mathcal{N}_{d,n-4}^{\Delta}$.
When $\mathcal{P}_{g,2n}\subset\mathcal{P}$, by using independent equations of the tangent cones $\mathcal{T}_{\Theta,x}$, we can provide an upper bound for the dimension of $\mathcal{P}$.
Combining this with the result from Section~\ref{165857_20Mar25},
we obtain
$\dim{\mathcal{P}}\leq\dim{\mathcal{P}_{g,2n}}$, which implies that
$\mathcal{P}=\mathcal{P}_{g,2n}$.
\section{Higher base loci and the tangent cone of theta divisors}\label{094948_19Mar25}
We assume that $n\geq1$ and $g\geq1$.
Let $(P,[\mathcal{L}])$ be the Prym variety of $[\tilde{C}\overset{\phi}{\rightarrow}{C}]\in\mathcal{R}_{g,2n}$.
Following Mumford's paper \cite{M2}, we describe the base locus of the linear system $|\mathcal{L}|$.
By \cite[p333]{M2}, we may assume that
$$
[-1]_{J_{C}}^{*}\mathcal{O}_{J_{C}}(\Theta_{C})
=\mathcal{O}_{J_{C}}(\Theta_{C}),\quad
[-1]_{J_{\tilde{C}}}^{*}\mathcal{O}_{J_{\tilde{C}}}(\Theta_{\tilde{C}})
=\mathcal{O}_{J_{\tilde{C}}}(\Theta_{\tilde{C}}),
$$
and
$$
\psi^{*}\mathcal{O}_{J_{\tilde{C}}}(\Theta_{\tilde{C}})
=\mathcal{O}_{J_{C}}(2\Theta_{C}),\quad
\mathcal{L}=\mathcal{O}_{\tilde{C}}(\Theta_{\tilde{C}})\vert_{P},
$$
where $\psi:J_{C}\rightarrow{J_{\tilde{C}}}$ is the homomorphism defined by
$$
J_{C}=\Pic^{0}{(C)}\longrightarrow{J_{\tilde{C}}=\Pic^{0}{(\tilde{C})}};\,
\xi\longmapsto\phi^{*}\xi.
$$
We define the homomorphism
$$
\pi:\,J_{C}\times{P}\longrightarrow{J_{\tilde{C}}};\,
(\xi,\mathcal{F})\longmapsto{\mathcal{F}\otimes\phi^{*}\xi}.
$$
Then we have
$\pi^{*}\mathcal{O}_{J_{\tilde{C}}}(\Theta_{\tilde{C}})
={\mathrm{pr}_{1}^{*}\mathcal{O}_{J_{C}}(2\Theta_{C})
\otimes{\mathrm{pr}_{2}^{*}\mathcal{L}}}$,
where $\mathrm{pr}_{i}$ denotes the $i$-th projection of $J_{C}\times{P}$.
Let
$s\in{H^{0}(J_{\tilde{C}},\mathcal{O}_{J_{\tilde{C}}}(\Theta_{\tilde{C}}))}$
be a nontrivial section.
Since $n\geq1$, $\psi:J_{C}\rightarrow{J_{\tilde{C}}}$ is injective, and by \cite[p335]{M2}, we have
$$
\pi^{*}s
=\sum_{i=1}^{2^{g}}\mathrm{pr}_{1}^{*}\alpha_{i}
\otimes{\mathrm{pr}_{2}^{*}\beta_{i}}
\in{H^{0}(J_{C}\times{P},
\mathrm{pr}_{1}^{*}\mathcal{O}_{J_{C}}(2\Theta_{C})
\otimes{\mathrm{pr}_{2}^{*}\mathcal{L}})}
$$
for some bases
$\alpha_{1},\dots,\alpha_{2^{g}}\in{H^{0}(J_{C},\mathcal{O}_{J_{C}}(2\Theta_{C}))}$
and
$\beta_{1},\dots,\beta_{2^{g}}\in{H^{0}(P,\mathcal{L})}$.
This implies the following proposition for the base locus of $|\mathcal{L}|$.
For $\mathcal{F}_{0}\in\Pic{(\tilde{C})}$, we denote the translation by
$$
t_{\mathcal{F}_{0}}:\,\Pic{(\tilde{C})}\longrightarrow\Pic{(\tilde{C})};\
\mathcal{F}\longmapsto{\mathcal{F}\otimes{\mathcal{F}_{0}}}.
$$
\begin{proposition}[\cite{M2} p334]\label{105542_26Feb25}
 The base locus of the linear system
 $|\mathcal{L}|$
 is
 $$
 \Bs{|\mathcal{L}|}=
 \{\mathcal{F}\in{P}\mid{\psi(J_{C})\subset
 t_{\mathcal{F}}^{-1}(\Theta_{\tilde{C}})}\}
 =\bigcap_{\xi\in\Pic^{0}{(C)}}
 t_{\phi^{*}\xi}^{-1}(\Theta_{\tilde{C}})\vert_{P}.
 $$
 More precisely, $H^{0}(P,\mathcal{L})$ is generated by defining sections of the hypersurfaces
 $t_{\phi^{*}\xi}^{-1}(\Theta_{\tilde{C}})\vert_{P}
 \subset{P}$
 for $\xi\in\Pic^{0}{(C)}$.
\end{proposition}
For $[\tilde{C}\overset{\phi}{\rightarrow}{C}]\in\mathcal{R}_{g,2n}$, there exists a unique invertible sheaf $\eta_{\phi}\in\Pic^{n}{(C)}$ such that $\Omega_{\tilde{C}}^{1}=\phi^{*}(\Omega_{C}^{1}\otimes\eta_{\phi})$.
Let
$$
P_{\phi}=\Nm^{-1}(\Omega_{C}^{1}\otimes\eta_{\phi})\subset
\Pic^{\tilde{g}-1}{(\tilde{C})}
$$
be the fiber of the norm map at $\Omega_{C}^{1}\otimes\eta_{\phi}\in\Pic{(C)}$,
where $\tilde{g}=2g+n-1$ is the genus of $\tilde{C}$.
We set
$$
W_{\tilde{C}}=\{\mathcal{F}\in\Pic^{\tilde{g}-1}{(\tilde{C})}\mid
h^{0}(\tilde{C},\mathcal{F})\geq1\}
$$
and
$\mathcal{L}_{\phi}=\mathcal{O}_{\Pic^{\tilde{g}-1}{(\tilde{C})}}(W_{\tilde{C}})\vert_{P_{\phi}}$.
Then there exists $\mathcal{F}_{0}\in{P_{\phi}}$ such that
$$
\mathcal{O}_{J_{\tilde{C}}}(\Theta_{\tilde{C}})=
t_{\mathcal{F}_{0}}^{*}\mathcal{O}_{\Pic^{\tilde{g}-1}{(\tilde{C})}}(W_{\tilde{C}}),
$$
hence the translation $t_{\mathcal{F}_{0}}$ induces isomorphisms $P=\Ker{(\Nm)}\overset{\simeq}{\rightarrow}P_{\phi}$
and
$\Bs{|\mathcal{L}|}\overset{\simeq}{\rightarrow}\Bs{|\mathcal{L}_{\phi}|}$.
In \cite{NO1}, Naranjo and Ortega gave an explicit description for $\Bs{|\mathcal{L}_{\phi}|}$.
For $j\geq0$, we denote by $\tilde{C}^{(j)}$ the $j$-th symmetric product of $\tilde{C}$, which parameterizes effective divisors of degree $j$ on $\tilde{C}$.
We define the subvariety in $\Pic^{\tilde{g}-1}{(\tilde{C})}$ by
$$
B_{i}=\Image{(\tilde{C}^{(n-2i)}\times\Pic^{g-1+i}{(C)}
\longrightarrow\Pic^{\tilde{g}-1}{(\tilde{C})};\,
(D,\mathcal{G})\longmapsto
\mathcal{O}_{\tilde{C}}(D)\otimes\phi^{*}\mathcal{G})}
$$
for $2i\leq{n}$.
When $2i>n$, we set $B_{i}=\emptyset$.
\begin{proposition}[\cite{NO1} Proposition~1.6.]\label{165606_22Feb25}
 $\displaystyle{\Bs{|\mathcal{L}_{\phi}|}
 =B_{1}\cap{P_{\phi}}}$.
\end{proposition}
When we set
$W_{\xi}=t_{\phi^{*}\xi}^{-1}(W_{\tilde{C}})\vert_{P_{\phi}}
\in|\mathcal{L}_{\phi}|$ for $\xi\in\Pic^{0}{(C)}$,
as in Proposition~\ref{105542_26Feb25}
we have
$$
\displaystyle{\Bs{|\mathcal{L}_{\phi}|}
=\bigcap_{\xi\in\Pic^{0}{(C)}}W_{\xi}
=B_{1}\cap{P_{\phi}}},
$$
and $H^{0}(P_{\phi},\mathcal{L}_{\phi})$ is generated by defining sections of $W_{\xi}$ for $\xi\in\Pic^{0}{(C)}$.
We define
$S_{i}(P_{\phi},\mathcal{L}_{\phi})
=\bigcap_{\Theta\in|\mathcal{L}_{\phi}|}\Theta_{\geq{i}}$,
where $\Theta_{\geq{i}}$ denotes the higher multiplicity locus
$$
\Theta_{\geq{i}}=\{\mathcal{F}\in\Theta\mid\mult_{\mathcal{F}}{\Theta}\geq{i}\}
\subset{P_{\phi}}.
$$
By the next proposition, we have
$$
S(P,\mathcal{L})\overset{\simeq}{\longrightarrow}
S_{2}(P_{\phi},\mathcal{L}_{\phi})
=\bigcap_{\xi\in\Pic^{0}{(C)}}W_{\xi,\mathrm{sing}}
=B_{2}\cap{P_{\phi}}.
$$
\begin{proposition}\label{160112_5Mar25}
 When $i\geq1$,
 $$
 S_{i}(P_{\phi},\mathcal{L}_{\phi})
 =\bigcap_{\xi\in\Pic^{0}{(C)}}
 W_{\xi,\geq{i}}
 =B_{i}\cap{P_{\phi}}.
 $$
\end{proposition}
We prepare some lemmas to prove Proposition~\ref{160112_5Mar25}.
\begin{lemma}\label{160809_4Mar25}
 If $\mathcal{F}\in{B_{i}\cap{P_{\phi}}}$, then
 $\mult_{\mathcal{F}}{W_{\xi}}\geq{i}$
 for any $\xi\in\Pic^{0}{(C)}$.
\end{lemma}
\begin{proof}
 There exist $D\in\tilde{C}^{(n-2i)}$ and
 $\mathcal{G}\in\Pic^{g-1+i}{(\mathcal{G})}$ such that
 $\mathcal{F}=\mathcal{O}_{\tilde{C}}(D)\otimes\phi^{*}\mathcal{G}$.
 Since
 $$
 \mathcal{G}\otimes\xi\subset
 \phi_{*}\phi^{*}(\mathcal{G}\otimes\xi)
 \subset
 \phi_{*}(\mathcal{O}_{\tilde{C}}(D)\otimes
 \phi^{*}(\mathcal{G}\otimes\xi))
 =\phi_{*}(\mathcal{F}\otimes\phi^{*}\xi),
 $$
 by the Riemann's singularity theorem,
 we have
 $$
 \mult_{\mathcal{F}}{W_{\xi}}\geq
 \mult_{\mathcal{F}}{t_{\phi^{*}\xi}^{-1}(W_{\tilde{C}})}
 =h^{0}(\tilde{C},\mathcal{F}\otimes\phi^{*}\xi)
 \geq{h^{0}(C,\mathcal{G}\otimes\xi)}
 \geq{i}.
 $$
\end{proof}
The following Mumford's exact sequence is essential for the proof of Proposition~\ref{160112_5Mar25}.
\begin{lemma}[\cite{M2} p338]\label{144152_27Feb25}
 Let $D$ be an effective divisor on $\tilde{C}$.
 If $\phi^{*}p\nleq{D}$ for any $p\in{C}$, then there exists an exact sequence
 $$
 0\longrightarrow
 \mathcal{O}_{C}\longrightarrow
 \phi_{*}\mathcal{O}_{\tilde{C}}(D)\longrightarrow
 \Nm{(\mathcal{O}_{\tilde{C}}(D))}\otimes\eta_{\phi}^{\vee}
 \longrightarrow0
 $$
 of locally free sheaves on $C$.
\end{lemma}
In the following, we write $H^{i}(\mathcal{G})=H^{i}(C,\mathcal{G})$
for any sheaf $\mathcal{G}$ on $C$.
When $0<2i\leq{n}$,
for $E\in{C^{(n-2i)}}$ and $\eta\in\Pic^{n}{(C)}$, we define
$$
K_{i}(E,\eta)=\{\mathcal{G}\in\Pic^{g-1+i}{(C)}\mid
h^{0}(\mathcal{G})=
h^{0}(\Omega_{C}^{1}\otimes\eta\otimes\mathcal{O}_{C}(-E)
\otimes\mathcal{G}^{\vee})=i\}.
$$
For $\mathcal{G}\in{K_{i}(E,\eta)}$, we define the hypersurface
$\mathcal{V}_{E,\mathcal{G}}\subset{H^{0}(\Omega_{C}^{1}\otimes\eta)^{\vee}}$
of degree $i$ by
$$
\det{
\begin{pmatrix}
 \mu(s_{1},t_{1})& &\mu(s_{1},t_{i})\\
 &\cdots& \\
 \mu(s_{i},t_{1})& &\mu(s_{i},t_{i})\\
\end{pmatrix}
\in\Sym^{i}{H^{0}(\Omega_{C}^{1}\otimes\eta)},
}
$$
where $s_{1},\dots,s_{i}\in{H^{0}(\mathcal{G})}$ and
$t_{1},\dots,t_{i}\in{H^{0}(\Omega_{C}^{1}\otimes\eta\otimes\mathcal{O}_{C}(-E)
\otimes\mathcal{G}^{\vee})}$
are bases, and
$$
\mu:\,H^{0}(\mathcal{G})\times
H^{0}(\Omega_{C}^{1}\otimes\eta\otimes\mathcal{O}_{C}(-E)
\otimes\mathcal{G}^{\vee})
\longrightarrow
H^{0}(\Omega_{C}^{1}\otimes\eta)
$$
denotes the composition of the multiplication map and the injective homomorphism
$H^{0}(\Omega_{C}^{1}\otimes\eta\otimes\mathcal{O}_{C}(-E))\hookrightarrow
{H^{0}(\Omega_{C}^{1}\otimes\eta)}$.
We note that $\mathcal{V}_{E,\mathcal{G}}$ does not depend on the choice of
bases $s_{1},\dots,s_{i}$ and $t_{1},\dots,t_{i}$.
\begin{lemma}\label{122252_5Mar25}
 Let $D\in\tilde{C}^{(n-2i)}$ and $\mathcal{G}\in\Pic^{g-1+i}{(C)}$
 be such that
 $\mathcal{F}=\mathcal{O}_{\tilde{C}}(D)\otimes\phi^{*}\mathcal{G}
 \in{P_{\phi}}$.
 For $\xi\in\Pic^{0}{(C)}$, the following conditions are equivalent:
 \begin{enumerate}
  \item $\mult_{\mathcal{F}}{W_{\xi}}=i$.
  \item $\mathcal{G}\otimes\xi\in{K_{i}(\phi(D),\eta_{\phi})}$,
	and $\phi^{*}p\nleq{D}$ for any $p\in{C}$.
 \end{enumerate}
 If these conditions are satisfied, then the tangent cone
 $\mathcal{T}_{W_{\xi},\mathcal{F}}$
 of $W_{\xi}$ at $\mathcal{F}$ is
 $$
 \mathcal{V}_{\phi(D),\mathcal{G}\otimes\xi}
 \subset{H^{0}(\Omega_{C}^{1}\otimes\eta_{\phi})^{\vee}}
 \simeq{T_{P_{\phi},\mathcal{F}}}.
 $$
\end{lemma}
\begin{proof}
 First, assume condition (1).
 If $\phi^{*}p\leq{D}$ for some $p\in{C}$, then $\mathcal{F}\in{B_{i+1}\cap{P_{\phi}}}$, and by Lemma~\ref{160809_4Mar25}, we have $\mult_{\mathcal{F}}{W_{\xi}}\geq{i+1}$, which contradicts the assumption that $\mult_{\mathcal{F}}{W_{\xi}}=i$.
 Thus, $\phi^{*}p\nleq{D}$ for any $p\in{C}$.
 Since
 $\Nm{(\mathcal{O}_{\tilde{C}}(D))}\otimes\mathcal{G}^{\otimes2}
 =\Omega_{C}^{1}\otimes\eta_{\phi}$,
 by Lemma~\ref{144152_27Feb25}, we obtain the exact sequence
 $$
 0\longrightarrow
 \mathcal{G}\otimes\xi
 \longrightarrow
 \phi_{*}(\mathcal{F}\otimes\phi^{*}\xi)
 \longrightarrow
 \Omega_{C}^{1}\otimes\mathcal{G}^{\vee}\otimes\xi
 \longrightarrow0.
 $$
 Since
 $$
 i=\mult_{\mathcal{F}}{W_{\xi}}\geq
 \mult_{\mathcal{F}}{t_{\phi^{*}\xi}^{-1}(W_{\tilde{C}})}
 =h^{0}(\tilde{C},\mathcal{F}\otimes\phi^{*}\xi)
 \geq
 h^{0}(\mathcal{G}\otimes\xi)\geq{i},
 $$
 it follows that $h^{0}(\mathcal{G}\otimes\xi)=i$ and
 $h^{1}(\mathcal{G}\otimes\xi)=0$.
 Therefore, $\mathcal{G}\otimes\xi\in{K_{i}(\phi(D),\eta_{\phi})}$, because
 \begin{align*}
  &h^{0}(\Omega_{C}^{1}\otimes\eta_{\phi}\otimes
  \Nm{(\mathcal{O}_{\tilde{C}}(-D))}
  \otimes(\mathcal{G}\otimes\xi)^{\vee})
  =h^{0}(\mathcal{G}\otimes\xi^{\vee})\\
  =&h^{0}(\Omega_{C}^{1}\otimes\mathcal{G}^{\vee}\otimes\xi)+i
  =h^{0}(\tilde{C},\mathcal{F}\otimes\phi^{*}\xi)
  -h^{0}(\mathcal{G}\otimes\xi)+i
  =i.
 \end{align*}
 Next, assume condition (2).
 By using Lemma~\ref{144152_27Feb25}, we obtain
 $$
 \mult_{\mathcal{F}}{t_{\phi^{*}\xi}^{-1}(W_{\tilde{C}})}
 =h^{0}(\tilde{C},\mathcal{F}\otimes\phi^{*}\xi)
 =
 h^{0}(\mathcal{G}\otimes\xi)=i.
 $$
 Let $\sigma:\tilde{C}\rightarrow\tilde{C}$ be the covering involution of $\phi$.
 Since $\mathcal{F}\in{P_{\phi}}$, we have
 $$
 \Omega_{\tilde{C}}^{1}=\phi^{*}(\Omega_{C}^{1}\otimes\eta_{\phi})
 =\mathcal{F}\otimes\sigma^{*}\mathcal{F}
 =\mathcal{F}\otimes\mathcal{O}_{\tilde{C}}(\sigma^{*}D)
 \otimes\phi^{*}\mathcal{G}.
 $$
 The injective homomorphisms
 $$
 \mathcal{G}\otimes\xi
 \hookrightarrow
 \phi_{*}(\mathcal{O}_{\tilde{C}}(D))\otimes
 \mathcal{G}\otimes\xi
 =\phi_{*}(\mathcal{O}_{\tilde{C}}(D)\otimes
 \phi^{*}(\mathcal{G}\otimes\xi))
 =\phi_{*}(\mathcal{F}\otimes\phi^{*}\xi)
 $$
 and
 \begin{multline*}
  \Omega_{C}^{1}\otimes\eta_{\phi}\otimes
  \Nm{(\mathcal{O}_{\tilde{C}}(-D))}
  \otimes(\mathcal{G}\otimes\xi)^{\vee}
  =\mathcal{G}\otimes\xi^{\vee}\\
  \hookrightarrow
  \phi_{*}(\mathcal{O}_{\tilde{C}}(\sigma^{*}D))\otimes
  \mathcal{G}\otimes\xi^{\vee}
  =\phi_{*}(\mathcal{O}_{\tilde{C}}(\sigma^{*}D)\otimes
  \phi^{*}(\mathcal{G}\otimes\xi^{\vee}))
  =\phi_{*}(\Omega^{1}_{\tilde{C}}\otimes
  (\mathcal{F}\otimes\phi^{*}\xi)^{\vee})
 \end{multline*}
 induce the isomorphisms
 $$
 H^{0}(C,\mathcal{G}\otimes\xi)
 \overset{\simeq}{\longrightarrow}
 H^{0}(\tilde{C},\mathcal{F}\otimes\phi^{*}\xi)
 $$
 and
 $$
 H^{0}(C,\Omega_{C}^{1}\otimes\eta_{\phi}\otimes
 \Nm{(\mathcal{O}_{\tilde{C}}(-D))}
 \otimes(\mathcal{G}\otimes\xi)^{\vee})
 \overset{\simeq}{\longrightarrow}
 H^{0}(\tilde{C},\Omega^{1}_{\tilde{C}}\otimes
 (\mathcal{F}\otimes\phi^{*}\xi)^{\vee}).
 $$
 From the commutative diagram
 $$
 \begin{array}{ccc}
  H^{0}(\mathcal{G}\otimes\xi)\times
   H^{0}(\Omega_{C}^{1}\otimes\eta_{\phi}\otimes
   \Nm{(\mathcal{O}_{\tilde{C}}(-D))}
   \otimes(\mathcal{G}\otimes\xi)^{\vee})
   &\overset{\mu}{\longrightarrow}&H^{0}(\Omega_{C}^{1}\otimes\eta_{\phi})\\
  \simeq\downarrow\quad&\circlearrowleft&\quad\downarrow{j}\\
  H^{0}(\tilde{C},\mathcal{F}\otimes\phi^{*}\xi)\times
   H^{0}(\tilde{C},\Omega^{1}_{\tilde{C}}\otimes
   (\mathcal{F}\otimes\phi^{*}\xi)^{\vee})
   &\underset{\cup}{\longrightarrow}&H^{0}(\tilde{C},\Omega_{\tilde{C}}^{1}),\\
 \end{array}
 $$
 and by the Riemann-Kempf singularity theorem \cite[Theorem~2]{K},
 the tangent cone
 of $t_{\phi^{*}\xi}^{-1}(W_{\tilde{C}})$ at $\mathcal{F}$ is
 $$
 (j^{\vee})^{-1}(\mathcal{V}_{\phi(D),\mathcal{G}\otimes\xi})
 \subsetneq{H^{0}(\tilde{C},\Omega_{\tilde{C}}^{1})^{\vee}}
 \simeq{T_{\Pic^{\tilde{g}-1}{(\tilde{C})},\mathcal{F}}},
 $$
 where
 $j^{\vee}:\,H^{0}(\tilde{C},\Omega_{\tilde{C}}^{1})^{\vee}\rightarrow
 {H^{0}(\Omega_{C}^{1}\otimes\eta_{\phi})^{\vee}}$
 is the dual of the injective homomorphism
 $$
 j:\,H^{0}(\Omega_{C}^{1}\otimes\eta_{\phi})\longrightarrow
 H^{0}(\tilde{C},\phi^{*}(\Omega_{C}^{1}\otimes\eta_{\phi}))
 =H^{0}(\tilde{C},\Omega_{\tilde{C}}^{1}).
 $$
 Since
 $T_{P_{\phi},\mathcal{F}}\subset{H^{0}(\tilde{C},\Omega_{\tilde{C}}^{1})^{\vee}}$
 is isomorphic to $H^{0}(\Omega_{C}^{1}\otimes\eta_{\phi})^{\vee}$ via $j^{\vee}$, the tangent cone of $W_{\xi}$ at $\mathcal{F}$ is
 $$
 T_{P_{\phi},\mathcal{F}}\cap
 (j^{\vee})^{-1}(\mathcal{V}_{\phi(D),\mathcal{G}\otimes\xi})
 \simeq\mathcal{V}_{\phi(D),\mathcal{G}\otimes\xi}\subsetneq
 H^{0}(\Omega_{C}^{1}\otimes\eta_{\phi})^{\vee},
 $$
 and hence we conclude that $\mult_{\mathcal{F}}{W_{\xi}}=i$.
\end{proof}
\begin{proof}[Proof of Proposition~\ref{160112_5Mar25}]
 Let $\Theta\in|\mathcal{L}_{\phi}|$ be the hypersurface defined by a section $\beta\in{H^{0}(P_{\phi},\mathcal{L}_{\phi})}$.
 We have
 $\Theta_{\geq{i}}\supset
 \bigcap_{\xi\in\Pic^{0}{(C)}}W_{\xi,\geq{i}}$,
 because $\beta$ is a linear combination of defining sections of
 $W_{\xi}$ for $\xi\in\Pic^{0}{(C)}$.
 Hence, we have
 $S_{i}(P_{\phi},\mathcal{L}_{\phi})
 =\bigcap_{\xi\in\Pic^{0}{(C)}}
 W_{\xi,\geq{i}}$.\par
 By Lemma~\ref{160809_4Mar25}, we have
 $
 B_{i}\cap{P_{\phi}}\subset
 \bigcap_{\xi\in\Pic^{0}{(C)}}W_{\xi,\geq{i}}.
 $
 By induction on $i$, we prove
 $B_{i}\cap{P_{\phi}}=
 \bigcap_{\xi\in\Pic^{0}{(C)}}W_{\xi,\geq{i}}$.
 When $i=1$, this is true by Proposition~\ref{165606_22Feb25}.
 We assume $i\geq2$.
 By the induction assumption,
 $$
 \bigcap_{\xi\in\Pic^{0}{(C)}}W_{\xi,\geq{i}}\subset
 \bigcap_{\xi\in\Pic^{0}{(C)}}W_{\xi,\geq{i-1}}
 =B_{i-1}\cap{P_{\phi}}.
 $$
 When $n<2(i-1)$, we have defined that $B_{i-1}=B_{i}=\emptyset$, and it follows that
 $\bigcap_{\xi\in\Pic^{0}{(C)}}W_{\xi,\geq{i}}=\emptyset$
 and
 $B_{i}\cap{P_{\phi}}=\emptyset$.
 When $n\geq2(i-1)$, for
 $\mathcal{F}\in
 \bigcap_{\xi\in\Pic^{0}{(C)}}W_{\xi,\geq{i}}$,
 there exist
 $D\in\tilde{C}^{(n-2(i-1))}$ and $\mathcal{G}\in\Pic^{g-1+(i-1)}{(C)}$
 such that
 $$
 \mathcal{F}=\mathcal{O}_{\tilde{C}}(D)\otimes\phi^{*}\mathcal{G},\quad
 \Omega_{C}^{1}\otimes\eta_{\phi}=
 \Nm{(\mathcal{O}_{\tilde{C}}(D))}\otimes\mathcal{G}^{\otimes2}.
 $$
 Since $\deg{\mathcal{G}}=g+i-2$, we can find $\xi_{0}\in\Pic^{0}{(C)}$
 such that
 $h^{0}(\mathcal{G}\otimes\xi_{0})=i-1$
 and
 $h^{0}(\mathcal{G}\otimes\xi_{0}^{\vee})=i-1$.
 Then, we have
 $\mathcal{G}\otimes\xi_{0}\in{K_{i-1}(\phi(D),\eta_{\phi})}$.
 If $\mathcal{F}\notin{B_{i}}$, then
 $\phi^{*}p\nleq{D}$ for any $p\in{D}$, and
 by Lemma~\ref{122252_5Mar25}, we have
 $\mult_{\mathcal{F}}{W_{\xi_{0}}}=i-1$.
 This is a contradiction to
 $\mathcal{F}\in\bigcap_{\xi\in\Pic^{0}{(C)}}{W_{\xi,\geq{i}}}$,
 hence $\mathcal{F}\in{B_{i}}$.
\end{proof}
\begin{proposition}\label{160207_5Mar25}
 $\dim{(B_{i}\cap{P_{\phi}})}=n-2i$ for  $0<2i\leq{n}$.
 If $0<2i<n$, then $B_{i}\cap{P_{\phi}}$ is irreducible.
 If $0<2i=n$, then $\#{(B_{i}\cap{P_{\phi}})}=4^{g}$.
\end{proposition}
\begin{proof}
 When $0<2i\leq{n}$, we define
 $$
 Y_{i}=\{(D,\mathcal{G})\in\tilde{C}^{(n-2i)}\times\Pic^{g-1+i}{(C)}\mid
 \Nm{(\mathcal{O}_{\tilde{C}}(D))\otimes\mathcal{G}^{\otimes2}}
 =\Omega_{C}^{1}\otimes\eta_{\phi}\}.
 $$
 The projection
 $$
 g_{i}:\,Y_{i}\longrightarrow\tilde{C}^{(n-2i)};\,
 (D,\mathcal{G})\longmapsto{D}
 $$
 is a finite \'{e}tale covering of degree $4^{g}$,
 hence $Y_{i}$ is smooth.
 When $n=2i$,
 $$
 Y_{i}\longrightarrow{B_{i}\cap{P_{\phi}}};\,
 (0,\mathcal{G})\longmapsto\phi^{*}\mathcal{G}
 $$
 is bijective, hence
 $\#{(B_{i}\cap{P_{\phi}})}=\#{Y_{i}}=4^{g}$.
 Assume $n\geq2i+1$.
 Let
 $\gamma:[0,1]\rightarrow\tilde{C}$
 be a loop on $\tilde{C}$,
 and let $\iota$ be the morphism defined by
 $$
 \iota:\,C\longrightarrow{J_{C}=\Pic^{0}{(C)}};\,
 p\longmapsto\mathcal{O}_{C}(p-\phi(\gamma(0))).
 $$
 Then, there is a unique lift
 $\delta:[0,1]\rightarrow{H^{0}(\Omega_{C}^{1})^{\vee}}$
 of $\iota\circ\phi\circ\gamma$
 such that $\delta(0)=0$:
 $$
 \begin{array}{ccccc}
  [0,1]&\overset{\delta}{\longrightarrow}&H^{0}(\Omega_{C}^{1})^{\vee}&
   \longrightarrow&H^{0}(\Omega_{C}^{1})^{\vee}/H_{1}(C,\mathbb{Z})\\
  \gamma\downarrow\quad& &\circlearrowleft& &\quad\downarrow\simeq\\
  \tilde{C}&\underset{\phi}{\longrightarrow}&C&
   \underset{\iota}{\longrightarrow}&J_{C}.\\
 \end{array}
 $$
 For $t\in[0,1]$, we denote by
 $\xi(t)\in\Pic^{0}{(C)}$ the invertible sheaf corresponding to
 $\bigl[\frac{1}{2}\delta(t)\bigr]
 \in{H^{0}(\Omega_{C}^{1})^{\vee}/H_{1}(C,\mathbb{Z})}=J_{C}$.
 If $(D,\mathcal{G})=(\gamma(0)+D',\mathcal{G})\in{Y_{i}}$, then the path
 $$
 [0,1]\longrightarrow{Y_{i}};\,
 t\longmapsto(\gamma(t)+D',\mathcal{G}\otimes\xi(t)^{\vee})
 $$
 on $Y_{i}$ is the lift of the loop
 $$
 [0,1]\longrightarrow\tilde{C}^{(n-2i)};\,
 t\longmapsto\gamma(t)+D'
 $$
 on $\tilde{C}^{(n-2i)}$.
 Since
 $\phi_{*}:H_{1}(\tilde{C},\mathbb{Z})\rightarrow
 {H_{1}(C,\mathbb{Z})}$
 is surjective,
 $$
 H_{1}(\tilde{C},\mathbb{Z})\longrightarrow
 {J_{C,2}=\Ker{([2]_{J_{C}})}};\,
 [\gamma]\longmapsto\xi(1)=\Bigl[\frac{1}{2}\delta(1)\Bigr]
 $$
 is surjective.
 This means that $\pi_{1}(\tilde{C}^{(n-2i)},D)$ acts transitively on the fiber $g_{i}^{-1}(D)$,
 and it follows that $Y_{i}$ is connected.
 Since $Y_{i}$ is also smooth, it is irreducible.
 Let $U_{i}$ be the subset of $\tilde{C}^{(n-2i)}$ defined by
 $$
 U_{i}=\{D\in\tilde{C}^{(n-2i)}\mid
 h^{0}(\tilde{C},\mathcal{O}_{\tilde{C}}(D))=1,\,
 \text{$\phi^{*}p\nleq{D}$ for any $p\in{C}$}\}.
 $$
 Since $0<n-2i\leq2g+n-1=\tilde{g}$,
 it is a non-empty open subset of $\tilde{C}^{(n-2i)}$.
 For
 $(D_{1},\mathcal{G}_{1}),(D_{2},\mathcal{G}_{2})\in
 g_{i}^{-1}(U_{i})$,
 we have
 $$
 h^{0}(\Nm{(\mathcal{O}_{\tilde{C}}(D_{1}))}\otimes\eta_{\phi}^{\vee}
 \otimes\mathcal{G}_{1}\otimes\mathcal{G}_{2}^{\vee})
 =h^{0}(\Omega_{C}^{1}\otimes\mathcal{G}_{1}^{\vee}
 \otimes\mathcal{G}_{2}^{\vee})
 =0,
 $$
 because
 $\deg{(\Omega_{C}^{1}\otimes\mathcal{G}_{1}^{\vee}
 \otimes\mathcal{G}_{2}^{\vee})}=-2i<0$.
 If
 $\mathcal{O}_{\tilde{C}}(D_{1})\otimes\phi^{*}\mathcal{G}_{1}
 =\mathcal{O}_{\tilde{C}}(D_{2})\otimes\phi^{*}\mathcal{G}_{2}$,
 then by Lemma~\ref{144152_27Feb25}, we have
 $$
 h^{0}(\mathcal{G}_{1}\otimes\mathcal{G}_{2}^{\vee})
 =h^{0}(\tilde{C},\mathcal{O}_{\tilde{C}}(D_{1})
 \otimes\phi^{*}(\mathcal{G}_{1}\otimes\mathcal{G}_{2}^{\vee}))
 =h^{0}(\tilde{C},\mathcal{O}_{\tilde{C}}(D_{2}))=1,
 $$
 hence $\mathcal{G}_{1}\otimes\mathcal{G}_{2}^{\vee}=\mathcal{O}_{C}$
 and $D_{1}=D_{2}$.
 This means that the morphism
 $$
 Y_{i}\longrightarrow{B_{i}\cap{P}_{\phi}};\,
 (D,\mathcal{G})\longmapsto\mathcal{O}_{\tilde{C}}(D)\otimes\phi^{*}\mathcal{G}
 $$
 is birational.
 Hence, $B_{i}\cap{P_{\phi}}$ is irreducible, and
 $$
 \dim{(B_{i}\cap{P}_{\phi})}=\dim{Y_{i}}=n-2i.
 $$
\end{proof}
By Proposition~\ref{160112_5Mar25} and Proposition~\ref{160207_5Mar25},
we have the following:
\begin{corollary}\label{154114_14Mar25}
 Let $(P,[\mathcal{L}])$ be the Prym variety of $[\tilde{C}\overset{\phi}{\rightarrow}{C}]\in\mathcal{R}_{g,2n}$ for $n\geq1$ and $g\geq1$.
 \begin{enumerate}
  \item \begin{enumerate}
	 \item $\dim{\Bs{|\mathcal{L}|}}=n-2$ and
	       $\Bs{|\mathcal{L}|}$ is irreducible for $n\geq3$,
	 \item $\dim{\Bs{|\mathcal{L}|}}=0$ and
	       $\#{\Bs{|\mathcal{L}|}}=4^{g}$ for $n=2$,
	 \item $\Bs{|\mathcal{L}|}=\emptyset$ for $n=1$,
	\end{enumerate}
  \item \begin{enumerate}
	 \item $\dim{S(P,\mathcal{L})}=n-4$ and
	       $S(P,\mathcal{L})$ is irreducible for $n\geq5$,
	 \item $\dim{S(P,\mathcal{L})}=0$ and
	       $\#{S(P,\mathcal{L})}=4^{g}$ for $n=4$,
	 \item $S(P,\mathcal{L})=\emptyset$ for $1\leq{n}\leq3$.
	\end{enumerate}
 \end{enumerate}
\end{corollary}
\section{Quadric hypersurfaces containing projective curves}\label{165857_20Mar25}
Let $C$ be a projective smooth curve of genus $g$, and let $\eta$ be an invertible sheaf on $C$ of degree
$n\geq3$.
By \cite[p55, Corollary]{M1}, the invertible sheaf $\Omega_{C}^{1}\otimes\eta$
is very ample, and the multiplication map
$$
\rho_{i}:\,\Sym^{i}{H^{0}(\Omega_{C}^{1}\otimes\eta)}
\longrightarrow
H^{0}((\Omega_{C}^{1}\otimes\eta)^{\otimes{i}})
$$
is surjective for $i\geq1$.
When $0<{2i}\leq{n}$, we have defined the hypersurface
$\mathcal{V}_{E,\mathcal{G}}\subset{H^{0}(\Omega_{C}^{1}\otimes\eta)^{\vee}}
=\Spec{(\Sym{H^{0}(\Omega_{C}^{1}\otimes\eta)})}$
for $E\in{C^{(n-2i)}}$ and $\mathcal{G}\in{K_{i}(E,\eta)}$.
Let
$Q_{E,\mathcal{G}}\subset{\mathbb{P}(H^{0}(\Omega_{C}^{1}\otimes\eta))}
=\Proj{(\Sym{H^{0}(\Omega_{C}^{1}\otimes\eta)})}$
be the projectivization of
$\mathcal{V}_{E,\mathcal{G}}$.
Then we have
$$
Q_{E,\mathcal{G}}=\bigcup_{D\in|\mathcal{G}|}
\mathbb{P}(H^{0}(\Omega_{C}^{1}\otimes\eta)/V_{E+D})
=\bigcup_{D\in|\Omega_{C}^{1}\otimes\eta\otimes\mathcal{O}_{C}(-E)
\otimes\mathcal{G}^{\vee}|}
\mathbb{P}(H^{0}(\Omega_{C}^{1}\otimes\eta)/V_{E+D})
$$
in $\mathbb{P}(H^{0}(\Omega_{C}^{1}\otimes\eta))$,
where
$$
V_{E+D}=\Image{(H^{0}(\Omega_{C}^{1}\otimes\eta\otimes\mathcal{O}_{C}(-E-D))
\hookrightarrow{H^{0}(\Omega_{C}^{1}\otimes\eta)})}.
$$
Let
$$
\Phi=\Phi_{|\Omega_{C}^{1}\otimes\eta|}:\,
C\longrightarrow\mathbb{P}(H^{0}(\Omega_{C}^{1}\otimes\eta))
=\mathbb{P}^{d-1}
$$
be the closed immersion defined by the very ample linear system $|\Omega_{C}^{1}\otimes\eta|$.
When $i\geq2$,
the defining section of $Q_{E,\mathcal{G}}\subset{\mathbb{P}^{d-1}}$
is contained in
$$
\Ker{(\rho_{i})}\subset
\Sym^{i}{H^{0}(\Omega_{C}^{1}\otimes\eta)}
={H^{0}(\mathbb{P}^{d-1},\mathcal{O}_{\mathbb{P}^{d-1}}(i))},
$$
hence, by the commutative diagram
$$
\begin{array}{ccc}
 C&\overset{\Phi_{|(\Omega_{C}^{1}\otimes\eta)^{\otimes{i}}|}}
  {\longrightarrow}&
  \mathbb{P}(H^{0}((\Omega_{C}^{1}\otimes\eta)^{\otimes{i}}))\\
 \Phi\downarrow\quad&
  \circlearrowleft&\downarrow\rho_{i}^{\vee}\\
 \mathbb{P}(H^{0}(\Omega_{C}^{1}\otimes\eta))=\mathbb{P}^{d-1}&
  \underset{|\mathcal{O}_{\mathbb{P}^{d-1}}(i)|}{\longrightarrow}&
  \mathbb{P}(\Sym^{i}{H^{0}(\Omega_{C}^{1}\otimes\eta)})
  =\mathbb{P}(H^{0}(\mathcal{O}_{\mathbb{P}^{d-1}}(i))),\\
\end{array}
$$
we have
$\Phi(C)\subset{Q_{E,\mathcal{G}}}$.
\begin{proposition}\label{180553_6Mar25}
 Let $\Sigma\subset{C}$ be a finite subset.
 If $n\geq4$ and $C$ is an elliptic curve or a hyperelliptic curve,
 then $\Ker{(\rho_{2})}$ is generated by
 the defining quadratics of $Q_{E,\mathcal{G}}$ for
 $E\in{(C\setminus\Sigma)^{(n-4)}}$ and $\mathcal{G}\in{K_{2}(E,\eta)}$.
\end{proposition}
First we prove it in the case where $n=4$.
\begin{lemma}\label{113436_12Mar25}
 If $n=4$ and $C$ is an elliptic curve or a hyperelliptic curve,
 then $\Ker{(\rho_{2})}$ is generated by
 the defining quadratics of $Q_{0,\mathcal{G}}$ for
 $\mathcal{G}\in{K_{2}(0,\eta)}$.
\end{lemma}
\begin{proof}
 Since $C$ is an elliptic curve or a hyperelliptic curve, there exists
 an invertible sheaf $\mathcal{H}\in\Pic^{2}{(C)}$ such that
 $h^{0}(\mathcal{H})=2$.
 Let $U\subset{C^{g+1}}$ be the set of $(p_{1},\dots,p_{g+1})\in{C^{g+1}}$ satisfying the following conditions:
 \begin{enumerate}
  \item\label{155626_9Mar25}
       $p_{i}\neq{p_{j}}$ for $1\leq{i}<j\leq{g+1}$,
  \item\label{155636_9Mar25}
       $\mathcal{O}_{C}(p_{1}+\dots+p_{g+1})
	\in{K_{2}(0,\eta)}$,
  \item\label{155642_9Mar25}
       $\mathcal{H}\otimes
       \mathcal{O}_{C}(p_{1}+\dots+p_{g+1}-p_{i}-p_{j})
       \in{K_{2}(0,\eta)}$
       for $1\leq{i}<j\leq{g+1}$,
  \item\label{155647_9Mar25}
       $h^{0}(\mathcal{H}^{\vee}\otimes
       \mathcal{O}_{C}(p_{1}+\dots+p_{g+1}))=0$,
  \item\label{155653_9Mar25}
       $h^{0}(\Omega_{C}^{1}\otimes\eta\otimes\mathcal{H}^{\vee}\otimes
       \mathcal{O}_{C}(-p_{1}-\dots-p_{g+1}))=0$.
 \end{enumerate}
 For $\mathcal{G}\in\Pic^{g+1}{(C)}$, the condition $\mathcal{G}\in{K(0,\eta)}$ is equivalent to
 $$
 h^{0}(\Omega_{C}^{1}\otimes\mathcal{G}^{\vee})=h^{0}(\mathcal{G}\otimes\eta^{\vee})=0.
 $$
 Since
 $$
 \deg{(\Omega_{C}^{1}\otimes\mathcal{G}^{\vee})}=\deg{(\mathcal{G}\otimes\eta^{\vee})}=g-3<g,
 $$
 $K_{2}(0,\eta)=\Pic^{g+1}{(C)}$ for $g\leq2$,
 and $K_{2}(0,\eta)$ is a non-empty open subset of $\Pic^{g+1}{(C)}$
 with $\dim{(\Pic^{g+1}{(C)}\setminus{K_{2}(0,\eta)})}=g-3$ for $g\geq3$.
 Hence the set of $(p_{1},\dots,p_{g+1})\in{C^{g+1}}$ satisfying (\ref{155636_9Mar25}) is a non-empty open subset of $C^{g+1}$.
 Since the image of the map
 $$
 C^{g+1}\longrightarrow\Pic^{g+1}{(C)};\,
 (p_{1},\dots,p_{g+1})\longmapsto
 \mathcal{H}\otimes
 \mathcal{O}_{C}(p_{1}+\dots+p_{g+1}-p_{i}-p_{j})
 $$
 has dimension $g-1$ for $1\leq{i}<j\leq{g+1}$,
 the set of $(p_{1},\dots,p_{g+1})\in{C^{g+1}}$ satisfying (\ref{155642_9Mar25}) is a non-empty open subset of $C^{g+1}$.
 Since (\ref{155626_9Mar25}), (\ref{155647_9Mar25}) and (\ref{155653_9Mar25}) are also open conditions,
 $U$ is a non-empty open subset of $C^{g+1}$.
 We fix $(p_{1},\dots,p_{g+1})\in{U}$, and set
 $\mathcal{G}_{0}=\mathcal{O}_{C}(p_{1}+\dots+p_{g+1})$
 and
 $\mathcal{G}_{ij}=\mathcal{H}\otimes
 \mathcal{O}_{C}(p_{1}+\dots+p_{g+1}-p_{i}-p_{j})$
 for $1\leq{i}<j\leq{g+1}$.
 We show that the defining quadratics of $Q_{0,\mathcal{G}}$ and
 $Q_{0,\mathcal{G}_{ij}}$
 form a basis of $\Ker{(\rho_{2})}$.
 For $E_{0}\neq{E_{\infty}}\in|\mathcal{H}|$, there is a rational function $f$ on $C$ such that $\Div{(f)}=E_{0}-E_{\infty}$.
 Let $\alpha_{ij}\in{H^{0}(\mathcal{G}_{ij})}$ be a section defining the divisor
 $E_{\infty}+p_{1}+\dots+p_{g+1}-p_{i}-p_{j}\in|\mathcal{G}_{ij}|$.
 Then $\alpha'_{ij}=f\alpha_{ij}\in{H^{0}(\mathcal{G}_{ij})}$ defines the divisor
 $E_{0}+p_{1}+\dots+p_{g+1}-p_{i}-p_{j}\in|\mathcal{G}_{ij}|$, and
 by (\ref{155642_9Mar25}),
 $\alpha_{ij},\alpha'_{ij}$ form a basis of $H^{0}(\mathcal{G}_{ij})$.
 We set
 $$
 V=\Image{(H^{0}(\Omega_{C}^{1}\otimes\eta
 \otimes\mathcal{O}_{C}(-E_{\infty}))
 \hookrightarrow{H^{0}(\Omega_{C}^{1}\otimes\eta)})}
 $$
 and
 $$
 V_{k}=\Image{(
 H^{0}(\Omega_{C}^{1}\otimes\eta\otimes\mathcal{O}_{C}(-E_{\infty}-p_{k}))
 \hookrightarrow
 H^{0}(\Omega_{C}^{1}\otimes\eta)
 )}
 $$
 for $1\leq{k}\leq{g+1}$.
 Since $\dim{V}=g+1$,
 by (\ref{155626_9Mar25}) and (\ref{155653_9Mar25}),
 we have
 $\bigcap_{k=1}^{g+1}V_{k}=\{0\}$,
 and $\dim{\bigl(\bigcap_{k\neq{i}}V_{k}\bigr)}=1$
 for $1\leq{i}\leq{g+1}$.
 Thus, there is a unique divisor
 $D_{i}\in|\Omega_{C}^{1}\otimes\eta\otimes\mathcal{H}^{\vee}
 \otimes\mathcal{O}_{C}(-p_{1}-\dots-p_{g+1}+p_{i})|$.
 By (\ref{155626_9Mar25}) and (\ref{155653_9Mar25}), we have
 $$
 D_{i}+p_{j}\neq{D_{j}+p_{i}}
 \in|\Omega_{C}^{1}\otimes\eta\otimes\mathcal{G}_{ij}^{\vee}|.
 $$
 for $1\leq{i}<j\leq{g+1}$.
 Let
 $\beta_{ij}\in{H^{0}(\Omega_{C}^{1}\otimes\eta\otimes\mathcal{G}_{ij}^{\vee})}$
 define $D_{i}+p_{j}$, and
 $\beta'_{ij}
 \in{H^{0}(\Omega_{C}^{1}\otimes\eta\otimes\mathcal{G}_{ij}^{\vee})}$
 define $D_{j}+p_{i}$.
 By (\ref{155642_9Mar25}),
 $\beta_{ij},\beta'_{ij}$ form a basis of
 $H^{0}(\Omega_{C}^{1}\otimes\eta\otimes\mathcal{G}_{ij}^{\vee})$.
 Let $\gamma_{i}$ be a generator of $\bigcap_{k\neq{i}}V_{k}$
 for $1\leq{i}\leq{g+1}$.
 Then $\gamma_{1},\dots,\gamma_{g+1}$ form a basis of $V$.
 Since
 $\mu(\alpha_{ij},\beta_{ij})
 \in\bigcap_{k\neq{i}}V_{k}$
 and
 $\mu(\alpha_{ij},\beta'_{ij})
 \in\bigcap_{k\neq{j}}V_{k}$,
 we have
 $\mu(\alpha_{ij},\beta_{ij})=c_{ij}\gamma_{i}$
 and
 $\mu(\alpha_{ij},\beta'_{ij})=c'_{ij}\gamma_{j}$
 for some $c_{ij},c'_{ij}\in\mathbb{C}^{*}$.
 By using the rational function $f$, we define the injective homomorphism
 $$
 \lambda:\,V\longrightarrow{H^{0}(\Omega_{C}^{1}\otimes\eta)};\,
 \gamma\longmapsto{f\gamma}.
 $$
 Then, by \cite[p195, Corollary~2]{AM},
 the defining quadratics
 \begin{align*}
  &\det{
  \begin{pmatrix}
   \mu(\alpha_{ij},\beta_{ij})&\mu(\alpha_{ij},\beta'_{ij})\\
   \mu(\alpha'_{ij},\beta_{ij})&\mu(\alpha'_{ij},\beta'_{ij})\\
  \end{pmatrix}}
  =
  \det{
  \begin{pmatrix}
   \mu(\alpha_{ij},\beta_{ij})&\mu(\alpha_{ij},\beta'_{ij})\\
   \lambda(\mu(\alpha_{ij},\beta_{ij}))&\lambda(\mu(\alpha_{ij},\beta'_{ij}))\\
  \end{pmatrix}}\\
  =
  &\det{
  \begin{pmatrix}
   c_{ij}\gamma_{i}&c'_{ij}\gamma_{j}\\
   \lambda(c_{ij}\gamma_{i})&\lambda(c'_{ij}\gamma_{j})\\
  \end{pmatrix}}
  =c_{ij}c'_{ij}
  \det{
  \begin{pmatrix}
   \gamma_{i}&\gamma_{j}\\
   \lambda(\gamma_{i})&\lambda(\gamma_{j})\\
  \end{pmatrix}}
  \in\Sym^{2}{H^{0}(\Omega_{C}^{1}\otimes\eta)}
 \end{align*}
 of $Q_{0,\mathcal{G}_{ij}}$ for $1\leq{i}<{j}\leq{g+1}$
 are linearly independent over $\mathbb{C}$.
 Since
 $$
 \dim{\Ker{(\rho_{2})}}
 =\frac{(g+3)(g+4)}{2}-(3g+5)
 =\frac{g(g+1)}{2}+1,
 $$
 the defining quadratics of $Q_{0,\mathcal{G}_{ij}}$
 for $1\leq{i}<j\leq{g+1}$
 generate a subspace of codimension $1$ in
 $\Ker{(\rho_{2})}$.\par
 The quadric hypersurface
 $Q_{0,\mathcal{G}_{ij}}\subset\mathbb{P}{(H^{0}(\Omega_{C}^{1}\otimes\eta))}$
 contains the line
 $$
 \ell_{V}
 =\mathbb{P}{(H^{0}(\Omega_{C}^{1}\otimes\eta)/V)}
 \subset\mathbb{P}{(H^{0}(\Omega_{C}^{1}\otimes\eta))}
 $$
 for $1\leq{i}<j\leq{g+1}$.
 It is enough to show that $\ell_{V}$ is not contained in
 $$
 Q_{0,\mathcal{G}_{0}}
 =\bigcup_{D\in|\mathcal{G}_{0}|}
 \mathbb{P}(H^{0}(\Omega_{C}^{1}\otimes\eta)/V_{D})
 \subset\mathbb{P}{(H^{0}(\Omega_{C}^{1}\otimes\eta))}.
 $$
 If $D\in|\mathcal{G}_{0}|$ satisfies
 $\supp{(E_{\infty})}\cap\supp{(D)}=\emptyset$,
 then, by (\ref{155653_9Mar25}), we have
 $V_{D}\cap{V}=\{0\}$ and
 $\ell_{V}\cap\mathbb{P}(H^{0}(\Omega_{C}^{1}\otimes\eta)/V_{D})
 =\emptyset$.
 By
 (\ref{155636_9Mar25}) and (\ref{155647_9Mar25}),
 since the pencil $|\mathcal{G}_{0}|$ has no base points,
 there exist at most two $D\in|\mathcal{G}_{0}|$
 such that $\supp{(E_{\infty})}\cap\supp{(D)}\neq\emptyset$.
 When $p\in\supp{(E_{\infty})}\cap\supp{(D)}$ and $E_{\infty}=p+p'$,
 by (\ref{155647_9Mar25}), we have $p'\nleq{D}$.
 If $V_{D}\subset{V}$, then
 $$
 2=h^{0}(\Omega_{C}^{1}\otimes\eta\otimes\mathcal{O}_{C}(-D))
 =h^{0}(\Omega_{C}^{1}\otimes\eta\otimes\mathcal{O}_{C}(-D-p')),
 $$
 and
 $h^{0}(\Omega_{C}^{1}\otimes\eta\otimes\mathcal{O}_{C}(-D-p-p'))\geq1$.
 But this contradicts (\ref{155653_9Mar25}),
 hence $V_{D}\nsubseteq{V}$ and
 $\ell_{V}\cap\mathbb{P}(H^{0}(\Omega_{C}^{1}\otimes\eta)/V_{D})
 =\{\Phi(p)\}$.
 This implies that $\ell_{V}\nsubseteq{Q_{0,\mathcal{G}_{0}}}$.
\end{proof}
\begin{proof}[Proof of Proposition~\ref{180553_6Mar25}]
 We prove it by induction on $n$.
 When $n=4$, the result holds by Lemma~\ref{113436_12Mar25}.
 We assume $n\geq5$.
 Let $p_{1},p_{2},p_{3}\in{C\setminus\Sigma}$ be three distinct points.
 We denote by
 $$
 \iota_{i}:\,
 \Sym^{2}{H^{0}(\Omega_{C}^{1}\otimes\eta\otimes\mathcal{O}_{C}(-p_{i}))}
 \hookrightarrow
 \Sym^{2}{H^{0}(\Omega_{C}^{1}\otimes\eta)}
 $$
 the natural injective homomorphism.
 Since the multiplication maps
 $$
 \rho_{2,p_{i}}:\,
 \Sym^{2}{H^{0}(\Omega_{C}^{1}\otimes\eta\otimes\mathcal{O}_{C}(-p_{i}))}
 \longrightarrow
 H^{0}((\Omega_{C}^{1}\otimes\eta\otimes\mathcal{O}_{C}(-p_{i}))^{\otimes2})
 $$
 and
 $$
 \rho_{2,p_{i}+p_{j}}:\,
 \Sym^{2}{H^{0}(\Omega_{C}^{1}\otimes\eta\otimes\mathcal{O}_{C}(-p_{i}-p_{j}))}
 \longrightarrow
 H^{0}((\Omega_{C}^{1}\otimes\eta\otimes
 \mathcal{O}_{C}(-p_{i}-p_{j}))^{\otimes2})
 $$
 are surjective, we have
 \begin{align*}
  &\dim{(\iota_{1}(\Ker{(\rho_{2,p_{1}})})+\iota_{2}(\Ker{(\rho_{2,p_{2}})}))}\\
  =&\dim{\iota_{1}(\Ker{(\rho_{2,p_{1}})})}
  +\dim{\iota_{2}(\Ker{(\rho_{2,p_{2}})})}
  -\dim{(\iota_{1}(\Ker{(\rho_{2,p_{1}})})\cap
  \iota_{2}(\Ker{(\rho_{2,p_{2}})}))}\\
  =&\dim{\Ker{(\rho_{2,p_{1}})}}+\dim{\Ker{(\rho_{2,p_{2}})}}
  -\dim{\Ker{(\rho_{2,p_{1}+p_{2}})}}
  =\dim{\Ker{(\rho_{2})}}-1.
 \end{align*}
 Let
 $$
 \pi_{i}:\,
 \mathbb{P}(H^{0}(\Omega_{C}^{1}\otimes\eta))\setminus\{\Phi(p_{i})\}
 \longrightarrow
 \mathbb{P}(H^{0}(\Omega_{C}^{1}\otimes\eta\otimes\mathcal{O}_{C}(-p_{i})))
 $$
 denote the projection from $\Phi(p_{i})$,
 and let
 $$
 \Phi_{i}=\Phi_{|\Omega_{C}^{1}\otimes\eta\otimes\mathcal{O}_{C}(-p_{i})|}:\,
 C\longrightarrow
 \mathbb{P}(H^{0}(\Omega_{C}^{1}\otimes\eta\otimes\mathcal{O}_{C}(-p_{i})))
 $$
 be the closed immersion defined by the very ample linear system
 $|\Omega_{C}^{1}\otimes\eta\otimes\mathcal{O}_{C}(-p_{i})|$.
 Let $\ell\subset\mathbb{P}(H^{0}(\Omega_{C}^{1}\otimes\eta))$ be the line through the points $\Phi(p_{1})$ and $\Phi(p_{2})$.
 Since
 $$
 \pi_{1}(\ell)=\pi_{1}(\Phi(p_{2}))
 =\Phi_{1}(p_{2})\neq\Phi_{1}(p_{3})
 =\pi_{1}(\Phi(p_{3})),
 $$
 we have $\Phi(p_{3})\notin\ell$.
 Hence, $\pi_{3}(\ell)$ is a line in
 $\mathbb{P}(H^{0}(\Omega_{C}^{1}\otimes\eta\otimes\mathcal{O}_{C}(-p_{3})))$,
 and we have
 $$
 \pi_{3}(\ell)\nsubseteq
 \Phi_{3}(C)\subset
 \mathbb{P}(H^{0}(\Omega_{C}^{1}\otimes\eta\otimes\mathcal{O}_{C}(-p_{3}))).
 $$
 By \cite{SD}, we have
 $\Phi_{3}(C)=\bigcap_{g\in\Ker{(\rho_{2,p_{3}})}}Q_{g}$,
 where
 $Q_{g}\subset
 \mathbb{P}(H^{0}(\Omega_{C}^{1}\otimes\eta\otimes\mathcal{O}_{C}(-p_{3})))$
 denotes the quadric defined by $g$.
 Thus, there exists $g\in\Ker{(\rho_{2,p_{3}})}$ such that
 $\pi_{3}(\ell)\nsubseteq{Q_{g}}$.
 Since the quadric
 $\overline{\pi_{3}^{-1}(Q_{g})}\subset
 \mathbb{P}(H^{0}(\Omega_{C}^{1}\otimes\eta))$
 defined by $\iota_{3}(g)\in\Ker{(\rho_{2})}$
 does not contain the line $\ell$,
 we have
 $\iota_{3}(g)\notin
 \iota_{1}(\Ker{(\rho_{2,p_{1}})})+\iota_{2}(\Ker{(\rho_{2,p_{2}})})$.
 Hence,
 $$
 \Ker{(\rho_{2})}=
 \iota_{1}(\Ker{(\rho_{2,p_{1}})})+\iota_{2}(\Ker{(\rho_{2,p_{2}})})
 +\iota_{3}(\Ker{(\rho_{2,p_{3}})}).
 $$
 By the induction assumption,
 $\Ker{(\rho_{2,p_{i}})}\subset
 \Sym^{2}{H^{0}(\Omega_{C}^{1}\otimes\eta\otimes\mathcal{O}_{C}(-p_{i}))}$
 is generated by the defining quadratics of
 $Q_{E,\mathcal{F}}
 \subset\mathbb{P}(H^{0}(\Omega_{C}^{1}\otimes\eta
 \otimes\mathcal{O}_{C}(-p_{i})))$
 for $E\in{(C\setminus\Sigma)^{(n-5)}}$ and
 $\mathcal{F}\in{K_{2}(E,\eta\otimes\mathcal{O}_{C}(-p_{i}))}$.
 Hence,
 $\iota_{i}(\Ker{(\rho_{2,p_{i}})})$
 is generated by the defining quadratics of
 $Q_{E+p_{i},\mathcal{F}}=\overline{\pi_{i}^{-1}(Q_{E,\mathcal{F}})}
 \subset\mathbb{P}(H^{0}(\Omega_{C}^{1}\otimes\eta))$
 for $E+p_{i}\in{(C\setminus\Sigma)^{(n-4)}}$ and
 $\mathcal{F}\in{K_{2}(E,\eta\otimes\mathcal{O}_{C}(-p_{i}))
 =K_{2}(E+p_{i},\eta)}$.
\end{proof}
\begin{corollary}\label{155343_14Mar25}
 Let $(P,[\mathcal{L}])$ be the Prym variety of $[\tilde{C}\overset{\phi}{\rightarrow}{C}]\in\mathcal{R}_{g,2n}$.
 If $n\geq4$ and $C$ is an elliptic curve or a hyperelliptic curve,
 then $\Ker{(\rho_{2})}$ is generated by the defining quadratics
 of the tangent cones $\mathcal{T}_{\Theta,x}$ for
 $x\in{S(P,\mathcal{L})}$
 and $\Theta\in|\mathcal{L}|$
 satisfying $\mult_{x}{\Theta}=2$.
\end{corollary}
\begin{proof}
 Since $S(P,\mathcal{L})\overset{\simeq}{\rightarrow}
 S_{2}(P_{\phi},\mathcal{L}_{\phi})$,
 by Proposition~\ref{180553_6Mar25}, it is enough to show that
 for $E\in{(C\setminus\Branch{(\phi)})^{(n-4)}}$ and $\mathcal{G}\in{K_{2}(E,\eta_{\phi})}$, the quadric
 $\mathcal{V}_{E,\mathcal{G}}\subset{H^{0}(\Omega_{C}^{1}\otimes\eta_{\phi})^{\vee}}$
 is the tangent cone $\mathcal{T}_{W_{\xi},\mathcal{F}}$ for some
 $\mathcal{F}\in{S_{2}(P_{\phi},\mathcal{L}_{\phi})}$
 and $\xi\in\Pic^{0}{(C)}$
 satisfying $\mult_{\mathcal{F}}{W_{\xi}}=2$.
 For $E\in{(C\setminus\Branch{(\phi)})^{(n-4)}}$, there exist $D\in\tilde{C}^{(n-4)}$ such that $\phi(D)=E$ and $\phi^{*}p\nleq{D}$ for any $p\in{C}$.
 Let $\delta\in\Pic^{g+1}{(C)}$ be an invertible sheaf such that
 $\delta^{\otimes2}=\Omega_{C}^{1}\otimes\eta_{\phi}\otimes\mathcal{O}_{C}(-E)$.
 For $\mathcal{G}\in{K_{2}(E,\eta_{\phi})}$, we set
 $\mathcal{F}=\mathcal{O}_{\tilde{C}}(D)\otimes\phi^{*}\delta$ and
 $\xi=\mathcal{G}\otimes\delta^{\vee}\in\Pic^{0}{(C)}$.
 Then, by Lemma~\ref{122252_5Mar25}, we have $\mult_{\mathcal{F}}{W_{\xi}}=2$,
 and
 $\mathcal{V}_{E,\mathcal{G}}$ is the tangent cone of $W_{\xi}$
 at $\mathcal{F}$.
\end{proof}
\section{Andreotti-Mayer loci}\label{165927_20Mar25}
Let
$$
\mathfrak{H}_{d}
=\{\tau\in\Mat_{d}{(\mathbb{C})}\mid\tau=\trans{\tau},\,
\im{(\tau)}>0\}
$$
be the Siegel upper half-space of degree $d=g+n-1$.
For the type of the polarization
$\Delta=(\underbrace{1,\dots,1}_{n-1},\underbrace{2,\dots,2}_{g})$,
we denote the diagonal matrix
$\diag{\Delta}\in\Mat_{d}{(\mathbb{Z})}$
by the same notation $\Delta$.
We define the action of $\mathbb{Z}^{d}\times\mathbb{Z}^{d}$ on
$\mathbb{C}^{d}\times\mathfrak{H}_{d}$
by
$$
(\mathbb{Z}^{d}\times\mathbb{Z}^{d})\times
(\mathbb{C}^{d}\times\mathfrak{H}_{d})
\longrightarrow\mathbb{C}^{d}\times\mathfrak{H}_{d};\,
(({\boldsymbol m}',{\boldsymbol m}''),({\boldsymbol z},\tau))
\longmapsto
({\boldsymbol z}+\tau{\boldsymbol m}'+\Delta{\boldsymbol m}'',\tau),
$$
and denote its quotient by $\mathcal{U}_{d}^{\Delta}$.
Then
$$
u:\,\mathcal{U}_{d}^{\Delta}\longrightarrow\mathfrak{H}_{d};\,
({\boldsymbol z},\tau)\longmapsto\tau
$$
is a proper morphism of complex manifolds, and the fiber $A_{\tau}=u^{-1}(\tau)$ at $\tau\in\mathfrak{H}_{d}$ is an abelian variety of dimension $d$.
Let $\mathcal{L}_{\tau}$ be the invertible sheaf on $A_{\tau}$ defined by the Hermitian form
$$
H:\,\mathbb{C}^{d}\times\mathbb{C}^{d}\longrightarrow\mathbb{C};\,
({\boldsymbol z},{\boldsymbol w})\longmapsto
\trans{{\boldsymbol z}}\,(\im{(\tau)})^{-1}\,\overline{\boldsymbol w}
$$
and the semi-character
$$
\chi_{0}:\,\tau\mathbb{Z}^{d}+\Delta\mathbb{Z}^{d}\longrightarrow
\mathbb{C}^{*};\,
\tau{\boldsymbol m}'+\Delta{\boldsymbol m}''\longmapsto
\exp{\bigl(\pi\sqrt{-1}\,\trans{\boldsymbol m}'\Delta{\boldsymbol m}''
\bigr)}.
$$
Then $(A_{\tau},[\mathcal{L}_{\tau}])$ is a polarized abelian variety of type $\Delta$.
The moduli space $\mathcal{A}_{d}^{\Delta}$ is the quotient of
$\mathfrak{H}_{d}$ by the action
$$
\Gamma_{\Delta}\times\mathfrak{H}_{d}\longrightarrow\mathfrak{H}_{d};\,
(
\begin{pmatrix}
 \alpha&\beta\\
 \gamma&\delta
\end{pmatrix},
\tau
)\longmapsto
(\alpha\tau+\beta)(\gamma\tau+\delta)^{-1},
$$
where $\Gamma_{\Delta}$ is the discrete subgroup of $\Sp_{2d}{(\mathbb{R})}$
defined by
$$
\Gamma_{\Delta}
=\left\{
\begin{pmatrix}
 \alpha&\beta\\
 \gamma&\delta
\end{pmatrix}
\in\Sp_{2d}{(\mathbb{R})}
\,\middle|\,
\begin{pmatrix}
 \alpha&\beta{\Delta^{-1}}\\
 \Delta\gamma&\Delta\delta{\Delta^{-1}}
\end{pmatrix}
\in\Mat_{2d}{(\mathbb{Z})}
\right\}.
$$
Following \cite{M3},
for ${\boldsymbol c}',{\boldsymbol c}''\in\frac{1}{2}\mathbb{Z}^{d}$
the theta function is defined by
$$
\theta
\begin{bmatrix}
 \trans{\boldsymbol c}'\\
 \trans{\boldsymbol c}''
\end{bmatrix}
({\boldsymbol z},\tau)
=\sum_{{\boldsymbol m}\in\mathbb{Z}^{d}}\exp{\Bigl(
\pi\sqrt{-1}\bigl(
\trans{({\boldsymbol m}+{\boldsymbol c}')}
\tau({\boldsymbol m}+{\boldsymbol c}')
+2\trans{({\boldsymbol m}+{\boldsymbol c}')}
({\boldsymbol z}+{\boldsymbol c}'')
\bigr)
\Bigr)}.
$$
Then, by \cite[Theorem~3.2.7. and Remark~8.5.3.]{BL},
$$
\theta_{\boldsymbol i}({\boldsymbol z},\tau)
=\theta
\begin{bmatrix}
 0&\cdots&0&i_{1}&\cdots&i_{g}\\
 0&\cdots&0&0&\cdots&0
\end{bmatrix}
({\boldsymbol z},\tau),\quad
\Bigl({\boldsymbol i}=
\begin{pmatrix}
 i_{1}\\
 \vdots\\
 i_{g}
\end{pmatrix}
\in{I=\Bigl\{0,\frac{1}{2}\Bigr\}^{g}}\Bigr)
$$
form a basis of $H^{0}(A_{\tau},\mathcal{L}_{\tau})$.
We consider the complex analytic space
$$
\mathcal{S}=\bigl\{({\boldsymbol z},\tau)\in\mathcal{U}_{d}^{\Delta}\mid
\theta_{\boldsymbol i}({\boldsymbol z},\tau)=0,\,
\frac{\partial\theta_{\boldsymbol i}}{\partial{z_{j}}}({\boldsymbol z},\tau)=0
\quad({\boldsymbol i}\in{I},\,1\leq{j}\leq{d})\bigr\},
$$
where $z_{j}$ denotes the coordinate of
${\boldsymbol z}=
\begin{pmatrix}
 z_{1}\\
 \vdots\\
 z_{d}
\end{pmatrix}
\in\mathbb{C}^{d}$.
Let
$\upsilon=u\vert_{\mathcal{S}}:\mathcal{S}\rightarrow\mathfrak{H}_{d}$
be the restriction of
$u:\mathcal{U}_{d}^{\Delta}\rightarrow\mathfrak{H}_{d}$.
Then the fiber of $\upsilon$ at $\tau\in\mathfrak{H}_{d}$ is
$$
\upsilon^{-1}(\tau)=S(A_{\tau},\mathcal{L}_{\tau})
=\bigcap_{\Theta\in|\mathcal{L}_{\tau}|}\Theta_{\mathrm{sing}}.
$$
\begin{lemma}\label{180214_14Mar25}
 If $({\boldsymbol a},\tau)\in\mathcal{S}$,
 then
 $\bigl({\boldsymbol a}+\tau
 \begin{pmatrix}
  {\boldsymbol 0}\\
  {\boldsymbol i}
 \end{pmatrix}
 +\Delta
 \begin{pmatrix}
  {\boldsymbol 0}\\
  {\boldsymbol j}
 \end{pmatrix},
 \tau\bigr)\in\mathcal{S}$
 for ${\boldsymbol i},{\boldsymbol j}\in{I}$.
\end{lemma}
\begin{proof}
 For ${\boldsymbol i},{\boldsymbol i}'\in{I}$, there exists
 ${\boldsymbol i}''\in{I}$ such that
 ${\boldsymbol i}+{\boldsymbol i}'-{\boldsymbol i}''\in\mathbb{Z}^{g}$.
 We set
 ${\boldsymbol c}=
 \begin{pmatrix}
  c_{1}\\
  \vdots\\
  c_{d}
 \end{pmatrix}
 =
 \begin{pmatrix}
  {\boldsymbol 0}\\
  {\boldsymbol i}
 \end{pmatrix}
 \in\frac{1}{2}\mathbb{Z}^{d}$
 and
 ${\boldsymbol d}=\Delta
 \begin{pmatrix}
  {\boldsymbol 0}\\
  {\boldsymbol j}
 \end{pmatrix}
 \in\mathbb{Z}^{d}$.
 Then we have
 $$
 \theta_{{\boldsymbol i}'}
 ({\boldsymbol z}+\tau{\boldsymbol c}+{\boldsymbol d},\tau)
 =\exp{\Bigl(
 \pi\sqrt{-1}(-\trans{\boldsymbol c}\tau{\boldsymbol c}
 -2\trans{\boldsymbol c}{\boldsymbol z}
 +4\trans{\boldsymbol i}'{\boldsymbol j})
 \Bigr)}
 \theta_{{\boldsymbol i}''}
 ({\boldsymbol z},\tau),
 $$
 and
 \begin{multline*}
  \frac{\partial{\theta_{{\boldsymbol i}'}}}{\partial{z_{j}}}
  ({\boldsymbol z}+\tau{\boldsymbol c}+{\boldsymbol d},\tau)\\
  =\exp{\Bigl(
  \pi\sqrt{-1}(-\trans{\boldsymbol c}\tau{\boldsymbol c}
  -2\trans{\boldsymbol c}{\boldsymbol z}
  +4\trans{\boldsymbol i}'{\boldsymbol j})
  \Bigr)}
  \bigl(-2\pi\sqrt{-1}\,c_{j}
  \theta_{{\boldsymbol i}''}
  ({\boldsymbol z},\tau)
  +
  \frac{\partial{\theta_{{\boldsymbol i}''}}}{\partial{z_{j}}}
  ({\boldsymbol z},\tau)\bigr)
 \end{multline*}
 for $1\leq{j}\leq{d}$.
 Since
 $\theta_{{\boldsymbol i}''}({\boldsymbol a},\tau)=0$
 and
 $\frac{\partial{\theta_{{\boldsymbol i}''}}}{\partial{z_{j}}}
 ({\boldsymbol a},\tau)=0$,
 we have
 $\theta_{{\boldsymbol i}'}
 ({\boldsymbol a}+\tau{\boldsymbol c}+{\boldsymbol d},\tau)=0$
 and
 $\frac{\partial{\theta_{{\boldsymbol i}'}}}{\partial{z_{j}}}
 ({\boldsymbol a}+\tau{\boldsymbol c}+{\boldsymbol d},\tau)=0$.
\end{proof}
For $m\geq0$, the degeneracy set (\cite[3.6.]{F})
$$
\mathcal{S}_{m}=\{({\boldsymbol z},\tau)\in\mathcal{S}\mid
\dim_{({\boldsymbol z},\tau)}{\upsilon^{-1}(\tau)}\geq{m}\}
$$
is an analytic subset in $\mathcal{S}$, and by the proper mapping theorem
(\cite[10.6.1.]{GR}),
$\mathcal{N}_{m}=\upsilon(\mathcal{S}_{m})$
is an analytic subset in $\mathfrak{H}_{d}$.
Let $\mathcal{N}$ be a local irreducible component of $\mathcal{N}_{m}$
at $\tau_{0}\in\mathcal{N}_{m}$.
We denote by $\Lambda(\mathcal{N})$ the set of all irreducible components
$\mathcal{M}$ of
$\upsilon^{-1}(\mathcal{N})\cap\mathcal{S}_{m}$ such that
$\upsilon(\mathcal{M})=\mathcal{N}$.
Since $\upsilon$ is proper, $\Lambda(\mathcal{N})$ is a non-empty finite set.
Following \cite{D}, we consider the subset
$\mathcal{S}(\tau_{0},\mathcal{N})
=\upsilon^{-1}(\tau_{0})\cap
\bigcup_{\mathcal{M}\in\Lambda(\mathcal{N})}\mathcal{M}$
in $\upsilon^{-1}(\tau_{0})$.
\begin{lemma}\label{123635_14Mar25}
 If $\mathcal{N}$ is nonsingular, then
 $$
 T_{\mathcal{N},\tau_{0}}\subset
 \Bigl(
 \sum_{1\leq{k}\leq{l}\leq{d}}
 \frac{\partial\theta_{\boldsymbol i}}{\partial\tau_{kl}}
 ({\boldsymbol a},\tau_{0})\,d\tau_{kl}
 \Bigr)^{\perp}
 \subset{T_{\mathfrak{H}_{d},\tau_{0}}}
 $$
 for ${\boldsymbol i}\in{I}$ and
 $({\boldsymbol a},\tau_{0})\in{\mathcal{S}(\tau_{0},\mathcal{N})}$, where
 $\tau_{kl}=\tau_{lk}$ denotes the coordinate of
 $\tau=
 \begin{pmatrix}
  \tau_{11}&&\tau_{1d}\\
  &\cdots& \\
  \tau_{d1}& &\tau_{dd}
 \end{pmatrix}
 \in\mathfrak{H}_{d}$.
\end{lemma}
\begin{proof}
 Since
 $({\boldsymbol a},\tau_{0})\in{\mathcal{S}(\tau_{0},\mathcal{N})}$,
 there exists
 $\mathcal{M}\in\Lambda(\mathcal{N})$ such that
 $({\boldsymbol a},\tau_{0})\in\mathcal{M}$.
 Then,
 $$
 U=\{({\boldsymbol z},\tau)\in\mathcal{M}\setminus
 \mathcal{M}_{\mathrm{sing}}\mid
 \text{$T_{\mathcal{M},({\boldsymbol z},\tau)}
 \overset{d\upsilon}{\rightarrow}T_{\mathcal{N},\tau}$
 is surjective}\}
 $$
 is a non-empty open subset of $\mathcal{M}$.
 For $({\boldsymbol z},\tau)\in{U}$, we can take a local coordinate ${\boldsymbol t}=(t_{1},\dots,t_{\dim{\mathcal{N}}})$ of $\mathcal{N}$ at $\tau$ and
 a local coordinate
 $({\boldsymbol s},{\boldsymbol t})$ of $U$ at $({\boldsymbol z},\tau)$
 such that
 $\upsilon\vert_{U}:U\rightarrow\mathcal{N}$ is given by
 $({\boldsymbol s},{\boldsymbol t})\mapsto{\boldsymbol t}$.
 Let
 $({\boldsymbol z}({\boldsymbol s},{\boldsymbol t}),
 \tau({\boldsymbol t}))\in{U}\subset\mathcal{S}$
 be the point corresponding to $({\boldsymbol s},{\boldsymbol t})$.
 Then we have
 $$
 \theta_{\boldsymbol i}({\boldsymbol z}
 ({\boldsymbol s},{\boldsymbol t}),\tau({\boldsymbol t}))=0,\quad
 \frac{\partial\theta_{\boldsymbol i}}{\partial{z_{j}}}
 ({\boldsymbol z}({\boldsymbol s},{\boldsymbol t}),\tau({\boldsymbol t}))=0
 \quad({\boldsymbol i}\in{I},\,1\leq{j}\leq{d}).
 $$
 For $1\leq{\nu}\leq\dim{\mathcal{N}}$, by the chain rule, we get
 \begin{multline*}
  \frac{\partial}{\partial{t_{\nu}}}\Bigl(
  \theta_{\boldsymbol i}({\boldsymbol z}
  ({\boldsymbol s},{\boldsymbol t}),\tau({\boldsymbol t}))
  \Bigr)\\
  =
  \sum_{j=1}^{d}
  \frac{\partial\theta_{\boldsymbol i}}{\partial{z_{j}}}
  ({\boldsymbol z}({\boldsymbol s},{\boldsymbol t}),\tau({\boldsymbol t}))
  \frac{\partial{z_{j}}}{\partial{t_{\nu}}}({\boldsymbol s},{\boldsymbol t})
  +\sum_{1\leq{k}\leq{l}\leq{d}}
  \frac{\partial\theta_{\boldsymbol i}}{\partial{\tau_{kl}}}
  ({\boldsymbol z}({\boldsymbol s},{\boldsymbol t}),\tau({\boldsymbol t}))
  \frac{\partial{\tau_{kl}}}{\partial{t_{\nu}}}({\boldsymbol t}),
 \end{multline*}
 which implies
 $$
 \sum_{1\leq{k}\leq{l}\leq{d}}
 \frac{\partial\theta_{\boldsymbol i}}{\partial{\tau_{kl}}}
 ({\boldsymbol z}({\boldsymbol s},{\boldsymbol t}),\tau({\boldsymbol t}))
 \frac{\partial{\tau_{kl}}}{\partial{t_{\nu}}}({\boldsymbol t})
 =0.
 $$
 This shows that
 $$
 \frac{\partial}{\partial{t_{\nu}}}
 =\sum_{1\leq{k}\leq{l}\leq{d}}
 \frac{\partial{\tau_{kl}}}{\partial{t_{\nu}}}({\boldsymbol t})
 \frac{\partial}{\partial{\tau_{kl}}}
 \in
 \Bigl(
 \sum_{1\leq{k}\leq{l}\leq{d}}
 \frac{\partial\theta_{\boldsymbol i}}{\partial{\tau_{kl}}}
 ({\boldsymbol z}({\boldsymbol s},{\boldsymbol t}),\tau({\boldsymbol t}))
 \,d\tau_{kl}
 \Bigr)^{\perp}
 \subset{T_{\mathfrak{H}_{d},\tau({\boldsymbol t})}}.
 $$
 Since the open subset $U\subset\mathcal{M}$ is contained in the analytic subset
 $$
 \mathcal{M}'
 =\left\{({\boldsymbol z},\tau)\in\mathcal{M}\,\middle|\,
 T_{\mathcal{N},\tau}\subset
 \Bigl(
 \sum_{1\leq{k}\leq{l}\leq{d}}
 \frac{\partial\theta_{\boldsymbol i}}{\partial\tau_{kl}}
 ({\boldsymbol z},\tau)\,d\tau_{kl}
 \Bigr)^{\perp}\right\},
 $$
 we have
 $({\boldsymbol a},\tau_{0})\in\mathcal{M}=\mathcal{M}'$.
\end{proof}
\begin{proposition}\label{184527_14Mar25}
 Assume that $\dim{\upsilon^{-1}(\tau_{0})}=m$.
 If there exist $c_{{\boldsymbol i},j}\in\mathbb{C}$ and
 $({\boldsymbol a}_{j},\tau_{0})\in\mathcal{S}(\tau_{0},\mathcal{N})$
 for ${\boldsymbol i}\in{I}$ and $1\leq{j}\leq{r}$
 such that cotangent vectors
 $$
 \sum_{1\leq{k}\leq{l}\leq{d}}
 \sum_{{\boldsymbol i}\in{I}}c_{{\boldsymbol i},j}
 \frac{\partial\theta_{\boldsymbol i}}{\partial\tau_{kl}}
 ({\boldsymbol a}_{j},\tau_{0})\,d\tau_{kl}
 \in{T_{\mathfrak{H}_{d},\tau_{0}}^{\vee}}=
 \bigoplus_{1\leq{k}\leq{l}\leq{d}}\mathbb{C}d\tau_{kl}
 \quad
 (1\leq{j}\leq{r})
 $$
 are linearly independent, then
 $\dim{\mathcal{N}}\leq\dim{\mathfrak{H}_{d}}-r$.
\end{proposition}
\begin{proof}
 For $1\leq{j}\leq{r}$, let
 $\mathcal{M}_{j}\in\Lambda(\mathcal{N})$ satisfy
 $({\boldsymbol a}_{j},\tau_{0})\in\mathcal{M}_{j}$.
 By the upper semi-continuity of the fiber dimension \cite[3.4.]{F}, there exist open neighborhoods
 $U_{j}\subset\mathcal{M}_{j}$ at
 $({\boldsymbol a}_{j},\tau_{0})$
 such that
 $$
 m=\dim{\upsilon^{-1}(\tau_{0})}\geq
 \dim_{({\boldsymbol a}_{j},\tau_{0})}
 {(\upsilon^{-1}(\tau_{0})\cap\mathcal{M})}
 \geq
 \dim_{({\boldsymbol b},\tau)}
 {(\upsilon^{-1}(\tau)\cap\mathcal{M})}
 \geq{m}
 $$
 for any $({\boldsymbol b},\tau)\in{U_{j}}$.
 Then, by \cite[3.10.]{F}, the restriction
 $\upsilon\vert_{U_{j}}:U_{j}\rightarrow\mathcal{N}$
 is an open mapping.
 By replacing $U_{j}$ with smaller open neighborhoods, we may assume that
 $$
 \sum_{1\leq{k}\leq{l}\leq{d}}
 \sum_{{\boldsymbol i}\in{I}}c_{{\boldsymbol i},j}
 \frac{\partial\theta_{\boldsymbol i}}{\partial\tau_{kl}}
 ({\boldsymbol b}_{j},\tau_{j})\,d\tau_{kl}
 \quad
 (1\leq{j}\leq{r})
 $$
 are linearly independent for any
 $\{({\boldsymbol b}_{j},\tau_{j})\}_{1\leq{j}\leq{r}}
 \in\prod_{j=1}^{r}{U_{j}}$.
 Since
 $\bigcap_{j=1}^{r}\upsilon(U_{j})$ is an open neighborhood of
 $\mathcal{N}$ at $\tau_{0}$,
 there exists a point $\tau'_{0}\in\bigcap_{j=1}^{r}\upsilon(U_{j})$
 such that $\mathcal{N}$ is nonsingular at $\tau'_{0}$.
 Then,
 $\mathcal{N}'=\mathcal{N}\setminus\mathcal{N}_{\mathrm{sing}}$ is a local irreducible component of $\mathcal{N}_{m}$ at $\tau_{0}'$.
 For $({\boldsymbol a}'_{j},\tau'_{0})\in\upsilon^{-1}(\tau'_{0})\cap{U_{j}}$,
 we have
 $({\boldsymbol a}'_{j},\tau'_{0})\in\mathcal{S}(\tau'_{0},\mathcal{N}')$
 and
 $$
 \sum_{1\leq{k}\leq{l}\leq{d}}
 \sum_{{\boldsymbol i}\in{I}}c_{{\boldsymbol i},j}
 \frac{\partial\theta_{\boldsymbol i}}{\partial\tau_{kl}}
 ({\boldsymbol a}'_{j},\tau'_{0})\,d\tau_{kl}
 \quad
 (1\leq{j}\leq{r})
 $$
 are linearly independent.
 By Lemma~\ref{123635_14Mar25},
 we conclude that $\dim{\mathcal{N}}\leq\dim{\mathfrak{H}_{d}}-r$.
\end{proof}
Let
$\pi:\mathfrak{H}_{d}\rightarrow\mathcal{A}_{d}^{\Delta}
=\Gamma_{\Delta}\backslash\mathfrak{H}_{d}$
be the quotient map.
By Corollary~\ref{154114_14Mar25},
we have
$\pi^{-1}(\Prym_{g,2n}{(\mathcal{R}_{g,2n})})\subset\mathcal{N}_{n-4}$
for $n\geq4$, and $\dim{\upsilon^{-1}(\tau_{0})}=n-4$ for any
$\tau_{0}\in\pi^{-1}(\Prym_{g,2n}{(\mathcal{R}_{g,2n})})$.
\begin{lemma}\label{184253_14Mar25}
 If $n\geq4$ and
 $\tau_{0}\in\pi^{-1}(\Prym_{g,2n}{(\mathcal{R}_{g,2n})})$, then
 $\mathcal{S}(\tau_{0},\mathcal{N})=\upsilon^{-1}(\tau_{0})$
 for any local irreducible component $\mathcal{N}$ of $\mathcal{N}_{n-4}$
 at $\tau_{0}$.
\end{lemma}
\begin{proof}
 When $n\geq5$, by Corollary~\ref{154114_14Mar25},
 $\upsilon^{-1}(\tau_{0})$ is irreducible of dimension $n-4$.
 Since
 $$
 n-4=
 \dim{\upsilon^{-1}(\tau_{0})}
 \geq
 \dim_{({\boldsymbol z},\tau_{0})}{(\upsilon^{-1}(\tau_{0})\cap\mathcal{M})}
 \geq{n-4}
 $$
 for any irreducible component $\mathcal{M}$ of $\upsilon^{-1}(\mathcal{N})\cap\mathcal{S}_{m}$ and $({\boldsymbol z},\tau_{0})\in\mathcal{M}$,
 we have
 $\upsilon^{-1}(\tau_{0})\subset\mathcal{M}$ for any
 $\mathcal{M}\in\Lambda(\mathcal{N})$, hence
 $\upsilon^{-1}(\tau_{0})\subset\mathcal{S}(\tau_{0},\mathcal{N})$.
 When $n=4$, by Lemma~\ref{180214_14Mar25},
 for
 $\mathcal{M}\in\Lambda(\mathcal{N})$ and
 ${\boldsymbol i},{\boldsymbol j}\in{I}$, there exists
 $\mathcal{M}_{{\boldsymbol i},{\boldsymbol j}}\in\Lambda(\mathcal{N})$
 such that
 $$
 \mathcal{M}_{{\boldsymbol i},{\boldsymbol j}}
 =\bigl\{({\boldsymbol z},\tau)\in\mathbb{C}^{d}\times\mathfrak{H}_{d}\mid
 ({\boldsymbol z}-\tau
 \begin{pmatrix}
  {\boldsymbol 0}\\
  {\boldsymbol i}
 \end{pmatrix}
 -\Delta
 \begin{pmatrix}
  {\boldsymbol 0}\\
  {\boldsymbol j}
 \end{pmatrix},
 \tau)\in\mathcal{M}\bigr\}.
 $$
 By Corollary~\ref{154114_14Mar25}, we have
 $\upsilon^{-1}(\tau_{0})\subset\mathcal{S}(\tau_{0},\mathcal{N})$.
\end{proof}
We define $\mathcal{R}_{1,2n}^{h}=\mathcal{R}_{1,2n}$, and for $g\geq2$,
$$
\mathcal{R}_{g,2n}^{h}
=\{[\tilde{C}\overset{\phi}{\rightarrow}C]\in\mathcal{R}_{g,2h}\mid
\text{$C$ is a hyperelliptic curve}\}.
$$
\begin{lemma}\label{161319_14Mar25}
 Let $r=\frac{d(d+1)}{2}-3g-2n+3$.
 For $\tau_{0}\in\pi^{-1}(\Prym_{g,2n}{(\mathcal{R}_{g,2n}^{h})})$,
 there exist
 $$
 ({\boldsymbol a}_{1},\tau_{0}),\dots,({\boldsymbol a}_{r},\tau_{0})
 \in\upsilon^{-1}(\tau_{0})=S(A_{\tau_{0}},\mathcal{L}_{\tau_{0}})
 $$
 and
 $$
 f_{1}({\boldsymbol z}),\dots,f_{r}({\boldsymbol z})
 \in\bigoplus_{{\boldsymbol i}\in{I}}
 \mathbb{C}\,\theta_{\boldsymbol i}({\boldsymbol z},\tau_{0})
 =H^{0}(A_{\tau_{0}},\mathcal{L}_{\tau_{0}})\quad
 $$
 such that quadratics
 $$
 \sum_{k=1}^{d}\sum_{l=1}^{d}
 \frac{\partial^{2}f_{1}}{\partial{z_{k}}\partial{z_{l}}}
 ({\boldsymbol a}_{1})z_{k}z_{l},\,\dots,\,
 \sum_{k=1}^{d}\sum_{l=1}^{d}
 \frac{\partial^{2}f_{r}}{\partial{z_{k}}\partial{z_{l}}}
 ({\boldsymbol a}_{r})z_{k}z_{l}
 \in\mathbb{C}[z_{1},\dots,z_{d}]
 $$
 are linearly independent.
\end{lemma}
\begin{proof}
 Since $r=\dim{\Ker{(\rho_{2})}}$, by Corollary~\ref{155343_14Mar25},
 there exist
 $$
 ({\boldsymbol a}_{1},\tau_{0}),\dots,({\boldsymbol a}_{r},\tau_{0})
 \in\upsilon^{-1}(\tau_{0})=S(A_{\tau_{0}},\mathcal{L}_{\tau_{0}})
 $$
 and
 $$
 \Theta_{1},\dots,\Theta_{r}\in|\mathcal{L}_{\tau_{0}}|
 $$
 such that $\mult_{({\boldsymbol a}_{j},\tau_{0})}{\Theta_{j}}=2$,
 and the defining quadratics of the tangent cones
 $\mathcal{T}_{\Theta_{j},({\boldsymbol a}_{j},\tau_{0})}$ for
 $1\leq{j}\leq{r}$ form a basis of $\Ker{(\rho_{2})}$.
 Let
 $$
 f_{j}({\boldsymbol z})
 \in\bigoplus_{{\boldsymbol i}\in{I}}
 \mathbb{C}\,\theta_{\boldsymbol i}({\boldsymbol z},\tau_{0})
 =H^{0}(A_{\tau_{0}},\mathcal{L}_{\tau_{0}})
 $$
 be a defining section of $\Theta_{j}\subset{A_{\tau_{0}}}$.
 Then the defining quadratic of the tangent cone
 $\mathcal{T}_{\Theta_{j},({\boldsymbol a}_{j},\tau_{0})}$
 is given by
 $\sum_{k=1}^{d}\sum_{l=1}^{d}
 \frac{\partial^{2}f_{j}}{\partial{z_{k}}\partial{z_{l}}}
 ({\boldsymbol a}_{r})z_{k}z_{l}
 \in\mathbb{C}[z_{1},\dots,z_{d}]$.
\end{proof}
\begin{lemma}[Heat equation, \cite{BL} Proposition~8.5.5.]\label{183022_14Mar25}
 For ${\boldsymbol i}\in{I}$ and $1\leq{k}\leq{l}\leq{d}$,
 $$
 \frac{\partial^{2}\theta_{\boldsymbol i}}{\partial{{z}_{k}}\partial{z_{l}}}
 ({\boldsymbol z},\tau)
 =2\pi\sqrt{-1}(1+\delta_{kl})
 \frac{\partial\theta_{\boldsymbol i}}{\partial{\tau_{kl}}}
 ({\boldsymbol z},\tau).
 $$
\end{lemma}
\begin{proof}[Proof of Theorem~\ref{165539_28Jan25}]
 Let $\mathcal{N}$ be a local irreducible component of $\mathcal{N}_{n-4}$
 at $\tau_{0}\in\pi^{-1}(\Prym_{g,2n}{(\mathcal{R}_{g,2n}^{h})})$.
 By Lemma~\ref{161319_14Mar25} and
 Lemma~\ref{183022_14Mar25}, there exist
 $({\boldsymbol a}_{j},\tau_{0})
 \in\upsilon^{-1}(\tau_{0})$
 and
 $f_{j}({\boldsymbol z})\in\bigoplus_{{\boldsymbol i}\in{I}}
 \mathbb{C}\,\theta_{\boldsymbol i}({\boldsymbol z},\tau_{0})$
 for $1\leq{j}\leq{r}$
 such that the following quadratics are linearly independent:
 $$
 \sum_{1\leq{k}\leq{l}\leq{d}}\frac{\partial{f_{1}}}
 {\partial{\tau_{kl}}}({\boldsymbol a}_{1})z_{k}z_{l},\dots,
 \sum_{1\leq{k}\leq{l}\leq{d}}\frac{\partial{f_{r}}}
 {\partial{\tau_{kl}}}({\boldsymbol a}_{r})z_{k}z_{l}
 \in\mathbb{C}[z_{1},\dots,z_{d}].
 $$
 Hence, the cotangent vectors
 $$
 \sum_{1\leq{k}\leq{l}\leq{d}}\frac{\partial{f_{1}}}
 {\partial{\tau_{kl}}}({\boldsymbol a}_{1})d\tau_{kl},\dots,
 \sum_{1\leq{k}\leq{l}\leq{d}}\frac{\partial{f_{r}}}
 {\partial{\tau_{kl}}}({\boldsymbol a}_{r})d\tau_{kl}
 \in\bigoplus_{1\leq{k}\leq{l}\leq{d}}\mathbb{C}d\tau_{kl}
 $$
 are linearly independent.
 By Lemma~\ref{184253_14Mar25} and Proposition~\ref{184527_14Mar25},
 we have
 $$
 \dim{\mathcal{N}}\leq\dim{\mathfrak{H}_{d}}-r=3g+2n-3.
 $$
 Since $\mathcal{R}_{g,2n}$ is irreducible by \cite[p138]{C},
 there exist an irreducible component
 $\mathcal{P}$
 of the generalized Andreotti-Mayer locus
 $\mathcal{N}_{d,n-4}^{\Delta}\subset\mathcal{A}_{d}^{\Delta}$
 such that $\Prym_{g,2n}{(\mathcal{R}_{g,2n})}\subset\mathcal{P}$.
 Then, we have $\pi^{-1}(\mathcal{P})\subset\mathcal{N}_{n-4}$, and
 $$
 \dim{\mathcal{P}}=\dim_{\tau_{0}}{\pi^{-1}(\mathcal{P})}\leq
 \dim_{\tau_{0}}{\mathcal{N}}\leq3g+2n-3,
 $$
 for some local irreducible component $\mathcal{N}$ of $\mathcal{N}_{n-4}$
 at $\tau_{0}\in\pi^{-1}(\Prym_{g,2n}{(\mathcal{R}_{g,2n}^{h})})$.
 On the other hand, by Theorem~\ref{185532_14Mar25}, we have
 $$
 \dim{\mathcal{P}}\geq\dim{\Prym_{g,2n}{(\mathcal{R}_{g,2n})}}=3g+2n-3.
 $$
 Hence, we conclude that
 $\mathcal{P}=\overline{\Prym_{g,2n}{(\mathcal{R}_{g,2n})}}$.
\end{proof}

\end{document}